\documentclass[12pt,a4paper,reqno]{amsart}

\usepackage[latin1]{inputenc}
\usepackage[english]{babel}
\usepackage[T1]{fontenc}

\usepackage{amsmath}
\usepackage{amsfonts}
\usepackage{amssymb}
\usepackage{amsthm}
\usepackage{graphicx}

\usepackage{mathtools} 

\usepackage{enumitem}

\usepackage[final]{showkeys} 

\usepackage{cases}

\usepackage[left=2cm,right=2cm,%
top=2cm,bottom=2cm]{geometry}

\usepackage{hyperref}
\hypersetup{%
colorlinks=true, %
linkcolor=blue, %
citecolor=blue, %
filecolor=blue, %
urlcolor=blue, %
pdfauthor={R. B. Assunção, 
           O. H. Miyagaki, and 
           W. W. Santos}, %
pdftitle={Titulo}, %
pdfsubject={Elliptic partial differential equations}, %
pdfkeywords={Quasilinear elliptic equations, 
p-Laplacian operator, 
variational methods,
multiple critical nonlinearities}, %
pdfproducer={Latex}, %
pdfcreator={ps2pdf}}

\newtheorem{teorema}{Theorem}[section]
\newtheorem{lema}[teorema]{Lemma}
\newtheorem{prop}[teorema]{Proposition}

\theoremstyle{remark}
\newtheorem{afirmativa}{Claim}

\allowdisplaybreaks

\title[Existence and multiplicity results for elliptic critical problems]{Existence and multiplicity results on a class of quasilinear elliptic problems with cylindrical singularities involving multiple critical exponents}

\author[Assun\c{c}\~{a}o]{R.\@ B.\@ Assun\c{c}\~{a}o}
\address{R.\@ B.\@ Assun\c{c}\~{a}o \hfill\break\indent
Departamento de Matem\'{a}tica\,---\,Universidade Federal de Minas Gerais, UFMG\hfill\break\indent
Av.~Ant\^{o}nio Carlos, 6627\,---\,CEP 30161-970\,---\,Belo Horizonte, MG, Brasil}
\email[Corresponding author]{ronaldo@mat.ufmg.br}

\author[Santos]{W.\@ W.\@ dos Santos}
\thanks{W.\@ W.\@ dos Santos was partially supported by CAPES/Reuni.}
\address{W.\@ W.\@ dos Santos \hfill\break\indent
Departamento de Ciências Exatas e Biológicas \hfill\break\indent
Universidade Federal de S\~{a}o Jo\~{a}o del-Rei
\,---\, Campus Sete Lagoas, CSL \hfill\break\indent
Rodovia MG 424\,-\,Km 47\,---\,CEP 35701-970\,---\,Sete Lagoas, MG, Brasil}
\email{weler@ufsj.edu.br}

\author[Miyagaki]{O.\@ H.\@ Miyagaki}
\thanks{O.\@ H.\@ Miyagaki was partially supported by CNPq/Brasil and INCTMAT/Brasil.}
\address{O.\@ H.\@ Miyagaki \hfill\break\indent
Departamento de Matem\'{a}tica\,---\,Universidade Federal de Juiz de Fora, UFJF \hfill\break\indent
Cidade Universit\'{a}ria\,---\,CEP 36036-330\,---\,Juiz de Fora, MG, Brasil} 
\email{ohmiyagaki@gmail.com}

\date{Belo Horizonte, \today}



\begin{document}

\begin{abstract}
This work deals with the existence of at least two positive solutions for the class of quasilinear elliptic equations with cylindrical singularities and multiple critical nonlinearities that can be written in the form
\begin{align*}
-\operatorname{div}\left[\frac{|\nabla u|^{p-2}}{|y|^{ap}}\nabla u\right]
-\mu\,\frac{u^{p-1}}{|y|^{p(a+1)}}
= h\,\frac{u^{p^*(a,b)-1}}{|y|^{bp^*(a,b)}}
+\lambda g\,\frac{u^{q-1}}{|y|^{cp^*(a,c)}}, 
\qquad (x,y) \in \mathbb{R}^{N-k}\times\mathbb{R}^k.
\end{align*}
We consider $N \geqslant 3$, 
$\lambda >0$, 
$p < k \leqslant N$,
$1<p<N$, 
$0 \leqslant \mu <\bar{\mu} \equiv \left\{[k-p(a+1)]/p\right\}^p$,
$0 \leqslant a < (k-p)/p$,
$a \leqslant b < a+1$, 
$a \leqslant c < a+1$, 
$1\leqslant q <p$,
$p^*(a,b)=Np/[N-p(a+1-b)]$, 
and 
$p^*(a,c) \equiv Np/[N-p(a+1-c)]$; 
in particular, if $\mu = 0$ we can include the cases 
$(k-p)/p \leqslant a < k(N-p)/Np$ and
$a < b< c<k(N-p(a+1))/p(N-k) < a+1$.
We suppose that 
$g\in L\sb{\alpha}\sp{r}(\mathbb{R}\sp{N})$,
where $r = p^*(a,c)/[p^*(a,c)-q]$ and $\alpha = c(p^*(a,c)-q)$,
is positive in a ball and that it can change sign; we also suppose that $h \in L^\infty(\mathbb{R}^k)$ and that it has a finite, positive limit $h_0$ at the origin and at infinity. To prove our results we use the Nehari manifold methods and we establish sufficient conditions to overcome the lack of compactness.
\end{abstract}

\maketitle

\section{Introduction and main results}
\label{sec:intmain}
In this work we study the existence of at least two positive solutions for a class of quasilinear elliptic equations with cylindrical singularities and multiple critical nonlinearities that can be written in the form
\begin{align}
\label{problemamultiplicidade}
-\operatorname{div}\left[\frac{|\nabla u|^{p-2}}{|y|^{ap}}\nabla u\right]
-\mu\,\frac{u^{p-1}}{|y|^{p(a+1)}}
= h\,\frac{u^{p^*(a,b)-1}}{|y|^{bp^*(a,b)}}
+\lambda g\,\frac{u^{q-1}}{|y|^{cp^*(a,c)}}, 
\qquad (x,y) \in \mathbb{R}^{N-k}\times\mathbb{R}^k.
\end{align}
We consider $N \geqslant 3$, 
$\lambda >0$, 
$p < k \leqslant N$,
$1<p<N$, 
$0 \leqslant \mu <\bar{\mu} \equiv \left\{[k-p(a+1)]/p\right\}^p$,
$0 \leqslant a < (k-p)/p$,
$a \leqslant b < a+1$, 
$a \leqslant c < a+1$, 
$1\leqslant q <p$,
$p^*(a,b)=Np/[N-p(a+1-b)]$, 
and 
$p^*(a,c) \equiv Np/[N-p(a+1-c)]$;
in particular, if $\mu = 0$ we can include the cases 
$(k-p)/p \leqslant a < k(N-p)/Np$ and
$a < b< c<k(N-p(a+1))/p(N-k) < a+1$.
We suppose that 
$g\in L\sb{\alpha}\sp{r}(\mathbb{R}\sp{N})$,
where $r = p^*(a,c)/[p^*(a,c)-q]$ and $\alpha = c(p^*(a,c)-q)$,
is positive in a ball and that it can change sign; we also suppose that $h \in L^\infty(\mathbb{R}^k)$ and that it has a finite, positive limit $h_0$ at the origin and at infinity.

This type of problem has some interest in the dynamics of some galaxies. In this case, the cylindrical symmetry of the weights in the differential operator and on the nonlinearities are motivated by the fact that some galaxies are axially symmetric; see Badiale and Tarantello~\cite{MR1918928}. Problem~\eqref{problemamultiplicidade} also appears in the models of several physical phenomena related to the equilibrium of the temperature in an anisotropic medium which is possibly a \lq perfect insulator\rq\  at some points and a \lq perfect conductor\rq\ at other points. Frequently, this problem also models the stationary solutions for the concentration of some substance in a fluid; see Dautray and Lions~\cite{MR1036731} and Ghergu and  R{\u{a}}dulescu~\cite{MR2865669}. The mathematical motivations to study problem~\eqref{problemamultiplicidade} are due to the fact that this problem generalizes some types of quasilinear elliptic equations with the $p$-Laplacian operator and also presents multiple critical nonlinearities with cylindrical weights, and this introduces several analytical difficulties regarding the proofs of existence and multiplicity results.

In what follows we present a very brief historical sketch for this class of problems with emphasis on existence and multiplicity results.
For the Laplacian operator in bounded domains without weights, that is, for $p=2$, $a=0$, $b=0$, $c=0$, $\mu = 0$,  $q=1$
$h \equiv 1$,  $\lambda = 1$, and $g$ in the dual of the function space $H_0^1(\Omega)$, 
Tarantello~\cite{MR1168304} proved the existence of at least two solutions to problem~\eqref{problemamultiplicidade}. 
Ambrosetti, Br\'{e}zis and Cerami~\cite{MR1276168} proved, among other results, that there exists
$\Lambda > 0$ such that problem~\eqref{problemamultiplicidade} has at least two solutions for every $\lambda \in (0,\Lambda)$, also considering $0<q<1<p \leqslant 2N/(N-2)$, $g\equiv 1$, and keeping the other parameters as indicated. In the case of the problem with a spherical singularity on the critical nonlinearity, that is, for $b\neq 0$, 
and for a homogeneous term represented by $\mu \neq 0$, we cite Bouchekif and Matallah~\cite{MR2480011}. 
For some more general cases of the functions $g$ and $h$, with $a=0$, $b=0$, $c \neq 0$, and $\mu = 0$, we mention Abdellaoui and Peral~\cite{MR1884469}; on the other hand, for $ b \neq 0$ and $c = 0$, we cite Hsu and Lin~\cite{MR2651376}. 
Still in the case of problems with spherical weights for which $k = N$, 
the functions $g$ and $h$ not necessarily constants but now for the $p$-Laplacian operator we mention Hsu~\cite{MR2559275}, who studied this class of problems with $a=0$, $b=0$, $c=0$, and $\mu = 0$. For these values of the parameters but with $\mu \neq 0$ we cite Kang, Wang and Wei~\cite{MR2733236}.
Ghoussoub and Yuan~\cite{MR1695021} also studied problem~\eqref{problemamultiplicidade} with a singularity in the critical term, that is, $b\neq 0$, but with the functions $g \equiv 1$ and $h$ a positive constant and obtained several existence results. 
Hsu~\cite{MR2802979} studied multiplicity of solutions for this same class of problems with $h \equiv 1$ and $ c \neq 0$;
afterwards, Hsu and Lin~\cite{MR2903809} introduced a singularity on the operator making $a \neq 0$.
Among other authors that also have studied this class of problems in bounded domains and proved existence and multiplicity results we refer the reader to 
Ferrero and Gazzola~\cite{MR1876652}, 
Cao and Han~\cite{MR2092869}, 
Chen~\cite{MR2072017},
Kang and Peng~\cite{MR2161873}, 
Ghoussoub and Robert~\cite{MR2210661}, and
Hsu and Lin~\cite{MR2525571}.

In the case of unbounded domains, we cite Ghergu and R{\u{a}}dulescu~\cite{MR2179582}, who studied problem~\eqref{problemamultiplicidade} for the Laplacian operator with spherical weights, that is, 
$k = N$, $p=2$ and $a \neq 0$. Furthermore, both authors used a hypothesis on the boundedness of the function $h$ and its behavior at infinity and at the origin, as well as a hypothesis on the integrability of the function $g$,
without the homogeneous term, that is, with $\mu = 0$, with a singularity on the critical nonlinearity, represented by $b\neq 0$, but without the singularity on the subcritical nonlinearity, that is, with $c=0$. 
Their multiplicity result for problem~\eqref{problemamultiplicidade} was proved with the help of Ekeland's variational principle and the mountain pass theorem without the Palais-Smale condition.
In the case of the $p$-Laplacian with $a=0$, $b=0$, $c = 0$, $\mu = 0$, and $h \equiv 1$ but $g$ not necessarily constant, we cite Cao, Li and Zhou~\cite{MR1313196}. 
For $c \neq 0$, $g \equiv 1$ but $h$ not necessarily constant, we mention Abdellaoui, Felli and Peral~\cite{MR2233146}.

All the above mentioned papers involve only spherical singularities, that is, $k = N$. For an example of a multiplicity result for problem~\eqref{problemamultiplicidade} with cylindrical singularities we cite Bouchekif and El Mokhtar~\cite{MR2801239},  who studied problem~\eqref{problemamultiplicidade} in the case $3 \leqslant k \leqslant N$, $p=2$, $a \neq 0$, $b\neq 0$, $c=0$, $\mu \neq 0$, $q=1$ and functions $g$ and $h$ not necessarily constants.

Now we state a result due to Maz'ya which plays a fundamental role in our work, since it allows the variational formulation of problem~\eqref{problemamultiplicidade}.

Let $1 \leqslant k \leqslant N$, 
$z=(x,y) \in \mathbb{R}\sp{N-k}\times\mathbb{R}\sp{k}$, $1 < p < N$, and 
$\mu < \bar{\mu} \equiv \left[(k-p(a+1))/p \right]^p$. Let also  
$a \leqslant b \leqslant a+1$ 
if $a \leqslant (k-p)/p$; 
in particular, if $\mu = 0$ we can include the cases 
$(k-p)/p \leqslant a < k(N-p)/Np$ and
$a < b< c<k(N-p(a+1))/p(N-k) < a+1$.
Then there exists a positive constant 
$K(N,p,\mu,a,b)>0$ such that
\begin{align}
\label{eq:maz}
\left(\int\sb{\mathbb{R}\sp{N}}  \frac{|u(z)|\sp{p^*(a,b)}}{|y|\sp{bp^*(a,b)}} \,dz \right)\sp{p/p^*(a,b)}
\leqslant  C
\left(\int\sb{\mathbb{R}\sp{N}}  \frac{|\nabla u(z)|^p}{|y|\sp{ap}} \,dz
-\mu\int_{\mathbb{R}^N}\frac{|u(z)|^p}{|y|^{p(a+1)}} \, dz\right)
\end{align}
for every function 
$u \in \mathcal{D}\sb{a}\sp{1,p}(\mathbb{R}\sp{N} \backslash \{|y|=0\})$,
where $p^*(a,b)=Np/\left[ N-p(a+1-b)\right]$.

For a proof of inequality~\eqref{eq:maz} in the case 
$\mu = 0$ we refer the reader to the book by 
Maz'ya~\cite[Section 2.1.6]{maz'ya}; for the case $\mu \neq 0$, see the paper by Secchi, Smets and Willem~\cite{MR1990020}, where the value
$\bar{\mu} \equiv \left[(k-p(a+1))/p \right]^p$ is determined. Recall that in the particular case $k=N$, inequality~\eqref{eq:maz} was proved by Caffarelli, Kohn and Nirenberg~\cite{MR768824}; see also Lin~\cite{MR864416} for an inequality involving higher order derivatives in the case $k=N$.

It is worth mentioning that the optimal constant can be defined by
\begin{align}
\label{problemagazzinimusina}
\displaystyle \frac{1}{K(N,p,\mu,a,b)} \equiv 
\inf\sb{\substack{u\in \mathcal{D}_a^{1,p}(\mathbb{R}^N\backslash \{|y|=0\})\\
u \not\equiv 0}} 
\frac{\displaystyle\int\sb{\mathbb{R}\sp{N}} \frac{|\nabla u(z)|\sp{p}}{|y|\sp{ap}}\,dz
-\mu\displaystyle\int\sb{\mathbb{R}\sp{N}} \frac{|u(z)|\sp{p}}{|y|\sp{p(a+1)}}\,dz}
{\left(\displaystyle\int\sb{\mathbb{R}\sp{N}} \frac{|u(z)|\sp{p^*(a,b)}}{|y|\sp{bp^*(a,b)}}\,dz\right)\sp{p/p^*(a,b)}}
\end{align}
which has an independent interest.

Using the ideas from the calculus of variations, the set in which we look for solutions to problem~\eqref{problemamultiplicidade} is the Sobolev space $\mathcal{D}\sb{a}\sp{1,p}(\mathbb{R}\sp{N} \backslash \{|y|=0\})$ defined as the completion of $C\sb{c}\sp{\infty}(\mathbb{R}\sp{N})$, the space of smooth functions with compact support, with respect to the norm defined by
\begin{align}
\label{eq:normgcil}
\|u\|  \equiv \left(\int_{\mathbb{R}\sp{N}}\frac{|\nabla u|^p}{|y|^{ap}}\, dz - \mu\int_{\mathbb{R}\sp{N}}\frac{|u|^p}{|y|^{p(a+1)}}\, dz\right)^{1/p}.
\end{align}

It is a well known fact that 
$\mathcal{D}\sb{a}\sp{1,p}(\mathbb{R}\sp{N} \backslash \{|y|=0\})$ is a reflexive Banach space and that its elements can be identified with measurable functions up to sets of measure zero. Moreover, due to Maz'ya's inequality~\eqref{eq:maz} the embedding
$\mathcal{D}\sb{a}\sp{1,p}(\mathbb{R}\sp{N} \backslash \{|y|=0\}) \hookrightarrow L\sb{b}\sp{p^*(a,b)}(\mathbb{R}\sp{N})$ is continuous, where  
$L\sb{b}\sp{p^*(a,b)}(\mathbb{R}\sp{N})$ 
stands for the Lebesgue space $L\sp{p^*(a,b)}(\mathbb{R}\sp{N}; |y|\sp{-bp^*(a,b)}dz)$ 
with weight $|y|\sp{-bp^*(a,b)}$ and norm defined by
\begin{align*}
\|u\|\sb{L\sb{b}\sp{p^*(a,b)}} & \equiv \left( \displaystyle\int\sb{\mathbb{R}\sp{N}} \frac{|u(z)|\sp{p^*(a,b)}}{|y|\sp{bp^*(a,b)}} \, dz \right)^{1/p^*(a,b)}.
\end{align*}

Using Maz'ya's inequality~\eqref{eq:maz} we can present a variational formulation to problem~\eqref{problemamultiplicidade}. Indeed, to investigate the existence of solutions to  the problem we study the existence of critical points for the energy functional
 $\varphi_\lambda\colon \mathcal{D}\sb{a}\sp{1,p}(\mathbb{R}\sp{N} \backslash \{|y|=0\}) \to \mathbb{R}$, associated in a natural way to problem~\eqref{problemamultiplicidade}, and defined by
\begin{align}\label{funcionalmultiplicidade}
\begin{split}
\varphi_\lambda(u) 
& =\frac{1}{p}\int\sb{\mathbb{R}^N}\frac{|\nabla u|^{p}}{|y|^{ap}}\, dz
-\frac{\mu}{p}\int\sb{\mathbb{R}^N}\frac{|u|^p}{|y|^{p(a+1)}}\, dz \\
& \qquad -\frac{1}{p^*(a,b)}\int\sb{\mathbb{R}^N}\frac{h|u|^{p^*(a,b)}}{|y|^{bp^*(a,b)}}\, dz
-\frac{\lambda}{q}\int\sb{\mathbb{R}^N}\frac{g|u|^{q}}{|y|^{cp^*(a,c)}}\, dz.
\end{split}
\end{align}
It is standard to verify that its G\^{a}teaux derivative is given by 
\begin{align*}
\begin{split}
\langle \varphi'_\lambda (u), v \rangle
& = \int\sb{\mathbb{R}^N}\frac{|\nabla u|^{p-2}}{|y|^{ap}} 
\langle \nabla u , \nabla v \rangle \, dz
-\mu \int\sb{\mathbb{R}^N}\frac{|u|^{p-2}}{|y|^{p(a+1)}} u v \, dz \\
& \qquad -\int\sb{\mathbb{R}^N}\frac{h|u|^{p^*(a,b)-2}}{|y|^{bp^*(a,b)}} u v \, dz
-\lambda \int\sb{\mathbb{R}^N}\frac{g|u|^{q-2}u}{|y|^{cp^*(a,c)}} u v\, dz
\end{split}
\end{align*}
for every function $v \in \mathcal{D}\sb{a}\sp{1,p}(\mathbb{R}\sp{N} \backslash \{|y|=0\})$. 
Moreover, critical points of this functional are weak solutions to problem~\eqref{problemamultiplicidade}.

We remark that in the case 
$\lambda = 1$, $g \equiv 1$, and $h \equiv 1$, 
if the subcritical term and its corresponding cylindrical singularity do not verify the dimensional balance given by Maz'ya's inequality, then problem~\eqref{problemamultiplicidade} does not have solution other than the trivial one. This occurs because there is an \lq asymptotic competition\rq\ between the energy carried by the two nonlinearities; and if one dominates the other, then the weakest one vanishes and we do not obtain nontrivial solutions to problem~\eqref{problemamultiplicidade}. For more details on this phenomenon see 
Filippucci, Pucci, and Robert~\cite{MR2498753}
for the case $p=2$, $a=0$, $b=0$, and $c \neq 0$, 
Xuan and Wang~\cite{MR2606810} for the $p$-Laplacian operator with spherical weights, and also
Assun\c{c}\~{a}o, Miyagaki and Santos~\cite{ams:exist} for the general $p$-Laplacian operator with cylindrical weights represented by $a\neq 0$, $b \neq 0$, $c\neq 0$ and $1 \leq k \leq N$. 
In contrast, the parameter $\lambda$ controls the perturbation of the subcritical nonlinearity and, accordingly, we can recover the solution to the problem.

Inspired by 
Bouchekif and El Mokhtar~\cite{MR2801239} regarding the class of problem and by Tarantello~\cite{MR1168304},  
by Cao, Li and Zhou~\cite{MR1313196}, 
by Ghergu and R{\u{a}}dulescu~\cite{MR2179582}, and by
Hsu~\cite{MR2559275} 
with respect to the several ideas in the proofs, we also obtain an existence and a multiplicity result. 

To state our first result we present additional hypotheses on the functions $g$ and $h$.
\begin{enumerate}[label=($g\sb{\arabic*}$),ref=$g\sb{\arabic*}$]
\item\label{hipoteseg1}
The function $g\colon \mathbb{R}^N \to \mathbb{R}$ is such that $g\in L\sb{\alpha}\sp{r}(\mathbb{R}\sp{N})$,
where $r = p^*(a,c)/[p^*(a,c)-q]$ and $\alpha = c(p^*(a,c)-q)$.
\end{enumerate}
\begin{enumerate}[label=($h\sb{\arabic*}$),ref=$h\sb{\arabic*}$]
\item\label{hipoteseh1}
The function $h\colon \mathbb{R}^k \to \mathbb{R}$ 
is such that $h \in L^\infty(\mathbb{R}^k)$.
\end{enumerate}

Our existence result reads as follows.

\begin{teorema}\label{teoremamultiplicidade1}
Consider the parameters in the already specified intervals. Suppose that the function 
$g \not\equiv 0$ verifies hypothesis~\eqref{hipoteseg1}; suppose also that the function $h \not\equiv 0$ is such that $h(y)>h_0>0$ for every $y\in \mathbb{R}^k$ and verifies hypothesis~\eqref{hipoteseh1}. Then there exists $\Lambda_1 > 0$ such that 
for every $\lambda \in (0, \Lambda_1)$
problem~\eqref{problemamultiplicidade} has at least one positive solution in the space $\mathcal{D}\sb{a}\sp{1,p}(\mathbb{R}\sp{N} \backslash \{|y|=0\})$. 
\end{teorema}

In our case the standard methods cannot be applied directly to prove that the functional $\varphi_\lambda$ has critical points because the differential operator  $-\operatorname{div}\left[|y|^{-ap}|\nabla u|^{p-2}\nabla u\right]-\mu |y|^{-p(a+1)}u^{p-1}$ 
is not uniformly elliptic. 
It is worth mentioning that the combination of two critical exponents creates several difficulties in the analysis of the problem. 
Hence, to overcome the difficulties that present themselves in the proofs of the results we have to make a detailed analysis of the energy level of the Palais-Smale sequences. 
Since the energy functional is not bounded from below in the space $\mathcal{D}\sb{a}\sp{1,p}(\mathbb{R}\sp{N} \backslash \{|y|=0\})$, 
to prove Theorem~\ref{teoremamultiplicidade1} we show in section~\ref{multiplicidadesecao1} that the energy functional is coercive and bounded from below in the Nehari manifold $\mathcal{N}_\lambda$ associated to the problem. Following up, we identify an interval for the parameter $\lambda$ that excludes possibly degenerated critical points for the functional, which allows us to describe the Nehari manifold as a disjoint union of two subsets, namely,  $\mathcal{N}_\lambda = \mathcal{N}_\lambda^+ \cup \mathcal{N}_\lambda^-$. 
Roughly speaking, these subsets contain the critical points of minima and maxima in $\mathcal{N}_\lambda$, respectively. After that, we define critical levels in these components and, using the hypotheses on the functions $g$ and $h$, we can verify that these levels have different signs. We remark that this is an important step to conclude that the solutions are nontrivial; 
in contrast, Bouchekif and El Mokhtar~\cite{MR2801239} obtain an nontrivial solution promptly from the very nature of their problem.
The next step is to show the existence of Palais-Smale sequences for the energy functional at both the critical levels; this is done through some technical lemmas which apply Ekeland's variational principle and the implicit function theorem. 
This result does not follow directly from the result by Bouchekif and El Mokhtar~\cite{MR2801239}; therefore, in our case we have to adapt some ideas by Hsu~\cite{MR2802979}.
In section~\ref{multiplicidadesecao2} we conclude the proof of the theorem by showing that in fact there exists at least one solution in $\mathcal{N}_\lambda^+$ and that in this case we have a critical point of local minimum in $\mathcal{N}_\lambda$.

To state our second result we present new hypotheses on the functions $g$ and $h$.

\begin{enumerate}[label=($g\sb{\arabic*}$),ref=$g\sb{\arabic*}$,start=2]
\item\label{hipoteseg2}
There exist $\nu_0>0$ and $r_0>0$ such that $g(z)\geqslant \nu_0$, for every $z \in B_{r_0}(0)$.
\end{enumerate}
\begin{enumerate}[label=($h\sb{\arabic*}$),ref=$h\sb{\arabic*}$,start=2]
\item\label{hipoteseh2}
$\lim_{|y|\to 0}h(y) = \lim_{|y| \to \infty} h(y) = h_0> 0$ and $h(y) \geqslant h_0$, for every $y \in \mathbb{R}^k$.
\end{enumerate}

Our multiplicity result reads as follows.
 
\begin{teorema}\label{teoremamultiplicidade2}
Consider the parameters in the already specified intervals. 
Suppose that the hypotheses~\eqref{hipoteseg1},~\eqref{hipoteseg2},~\eqref{hipoteseh1} and~\eqref{hipoteseh2} are verified. Then there exists $\Lambda_0 > 0$ such that for every  $\lambda \in (0, \Lambda_0)$ problem~\eqref{problemamultiplicidade} has at least two positive solutions in the space $\mathcal{D}\sb{a}\sp{1,p}(\mathbb{R}\sp{N} \backslash \{|y|=0\})$ . 
\end{teorema}

The crucial step in the proof of Theorem~\ref{teoremamultiplicidade2} consists in identifying an energy level $d_\lambda^*$ below which we can recover the compactness of the Palais-Smale sequences. To reach this objective we adapt the ideas by Ghergu and R{\u{a}}dulescu~\cite{MR2179582} to compute some limits of sequences of integrals involving the critical term. The next step is to establish appropriate conditions under which this level is not exceeded for sufficiently small perturbations. 
The proofs of these results do not follow directly from the work by 
Bouchekif and El Mokhtar~\cite{MR2801239}; 
in our setting, we adapt some ideas by Cao, Li and Zhou~\cite{MR1313196}. Finally, we conclude the proof of Theorem~\ref{teoremamultiplicidade2} by showing that under the given hypotheses there exists at least one solution 
to problem~\eqref{problemamultiplicidade} in $\mathcal{N}_\lambda^-$; and since this component is disjoint from $\mathcal{N}_\lambda^+$, this solution is different from the other one already proved in Theorem~\ref{teoremamultiplicidade1}.

\section{Auxiliary results}\label{multiplicidadesecao1}

We know that the functional $\varphi_\lambda$ is not bounded from bellow in $\mathcal{D}\sb{a}\sp{1,p}(\mathbb{R}\sp{N} \backslash \{|y|=0\})$. Therefore, we look for solutions to problem~\eqref{problemamultiplicidade} in a subset of $\mathcal{D}\sb{a}\sp{1,p}(\mathbb{R}\sp{N} \backslash \{|y|=0\})$ defined by
\begin{align*}
\mathcal{N}_\lambda = \left\{u \in \mathcal{D}\sb{a}\sp{1,p}(\mathbb{R}\sp{N}\backslash\{|y|=0\}) \colon u \not\equiv 0 
\textrm{ and }\langle\varphi'_\lambda(u),u \rangle=0\right\}
\end{align*}
and called Nehari manifold. 

We claim that $\mathcal{N}_\lambda \neq \emptyset$. 
Indeed, let the functional 
$\psi_\lambda \colon
\mathcal{D}\sb{a}\sp{1,p}(\mathbb{R}\sp{N} \backslash \{|y|=0\}) \to \mathbb{R}$ 
be given by
$\psi_\lambda(u) = \langle \varphi_\lambda(u),u \rangle$.
Since $p^*(a,b) > p$, it follows that $\psi_\lambda(tu) <0$ 
for $t$ big enough and $\psi_\lambda(tu) >0$ for 
$t$ small enough. Then there exists $t>0$ such that 
$\psi_\lambda(tu) =0$, that is, $tu \in \mathcal{N}$.

Direct computations show that 
$u\in \mathcal{N}_\lambda$ if, and only if,
\begin{align}\label{nehari1}
 \int\sb{\mathbb{R}^N}\frac{|\nabla u|^{p}}{|y|^{ap}}\, dz
-\mu\int\sb{\mathbb{R}^N}\frac{|u|^p}{|y|^{p(a+1)}}\, dz
 -\int\sb{\mathbb{R}^N}\frac{h|u|^{p^*(a,b)}}{|y|^{bp^*(a,b)}}\, dz
-\lambda\int\sb{\mathbb{R}^N}\frac{g|u|^{q}}{|y|^{cp^*(a,c)}}\, dz = 0.
\end{align}
We note that the Nehari manifold $\mathcal{N}_\lambda$ contains the set of solutions to problem~\eqref{problemamultiplicidade},
that is, the critical points of $\varphi_\lambda$ belong to $\mathcal{N}_\lambda$  and, as we will see in Lemma~\ref{bouchekiflema2.4}, 
local minimizers for $\varphi_\lambda$ in $\mathcal{N}_\lambda$ are critical points for $\varphi_\lambda$ in general.

The relation~\eqref{nehari1} allows us to obtain the following 
equivalent expressions for the value 
of the functional $\varphi_\lambda(u)$ in an element $u\in\mathcal{N}_\lambda$, namely,
\begin{align}\label{phineharig}
\varphi_\lambda(u) 
& =\left(\frac{1}{p}-\frac{1}{p^*(a,b)}\right)\|u\|^p
 -\lambda\left(\frac{1}{q}-\frac{1}{p^*(a,b)}\right)\int\sb{\mathbb{R}^N}\frac{g|u|^{q}}{|y|^{cp^*(a,c)}}\, dz\\ \label{phineharih}
& =-\left(\frac{1}{q}-\frac{1}{p}\right)\|u\|^p 
 +\left(\frac{1}{q}-\frac{1}{p^*(a,b)}\right)\int\sb{\mathbb{R}^N}\frac{h|u|^{p^*(a,b)}}{|y|^{bp^*(a,b)}}\, dz.
\end{align}

\begin{lema}\label{bouchekiflema2.3}
The functional $\phi_\lambda$ is coercive and bounded from below in the Nehari manifold $\mathcal{N}_\lambda$.
\end{lema}
\begin{proof}
To show that the functional $\varphi_\lambda$ is coercive in $\mathcal{N}_\lambda$, first we consider $u \in \mathcal{N}_\lambda$. 
Using H\"{o}lder's inequality and Maz'ya's embedding, from equation~\eqref{phineharig} we get
\begin{align}\label{bouchekif2.3A}
\varphi_\lambda(u) 
& \geqslant \left(\frac{1}{p}-\frac{1}{p^*(a,b)}\right) \|u\|^p
-\lambda\left(\frac{1}{q}-\frac{1}{p^*(a,b)}\right)\|g\|_{L\sb{\alpha}\sp{r}(\mathbb{R}\sp{N})}K(N,p,\mu,a,c)^\frac{q}{p}\|u\|^q.
\end{align}
Coercivity follows from inequality~\eqref{bouchekif2.3A} together with the fact that $1<q<p$.

To show that the funcional $\varphi_\lambda$ is bounded from below in the Nehari manifold, we define the function 
$f\sb{0} \colon \mathbb{R}_*^+ \to \mathbb{R}$ 
by
\begin{align*}
f_0(t) 
\equiv \left(\frac{1}{p}-\frac{1}{p^*(a,b)}\right) t^p
-\lambda\left(\frac{1}{q}-\frac{1}{p^*(a,b)}\right)\|g\|_{L\sb{\alpha}\sp{r}(\mathbb{R}\sp{N})}K(N,p,\mu,a,c)^\frac{q}{p}t^q.
\end{align*}
The minimum of $f\sb{0}$ is 
$f_0(t_{\min}) =-\lambda^\frac{p}{p-q}C_0$,
where
\begin{align*}
C_0 \equiv \left(1-\frac{q}{p}\right)\left(\frac{1}{p}-\frac{1}{p^*(a,b)}\right)^{-\frac{q}{p-q}}
\left(\frac{1}{q}-\frac{1}{p^*(a,b)}\right)^{-\frac{p}{p-q}} \|g\|_{L\sb{\alpha}\sp{r}(\mathbb{R}\sp{N})}^\frac{p}{p-q}K(N,p,\mu,a,c)^\frac{q}{p-q}\left(\frac{q}{p}\right)^\frac{q}{p-q}.
\end{align*}
Hence, returning to inequality~\eqref{bouchekif2.3A} 
and using the previous result, we deduce that
\begin{align}\label{bouchekif2.3B}
\varphi_\lambda(u) \geqslant -\lambda^\frac{p}{p-q}C_0.
\end{align}
Since the function $u\in \mathcal{N}_\lambda$ is arbitrary, we obtain the boundedness from below of the functional $\varphi_\lambda$ in the Nehari manifold $\mathcal{N}_\lambda$.
\end{proof}

Thus, for $u\in \mathcal{N}_\lambda$ we can 
use the definition of the function \( \psi_\lambda\) and equation~\eqref{nehari1} to obtain the equalities
\begin{align}
\langle \psi_\lambda'(u),u\rangle
&= p\|u\|^p -p^*(a,b)\int\sb{\mathbb{R}^N}\frac{h|u|^{p^*(a,b)}}{|y|^{bp^*(a,b)}}\, dz
-\lambda q\int\sb{\mathbb{R}^N}\frac{g|u|^{q}}{|y|^{cp^*(a,c)}}\, dz\label{bouchekif2.4A}\\
&=(p-q)\|u\|^p - (p^*(a,b)-q)\int\sb{\mathbb{R}^N}\frac{h|u|^{p^*(a,b)}}{|y|^{bp^*(a,b)}}\, dz\label{bouchekif2.4B}\\
&=-(p^*(a,b)-p)\|u\|^p+\lambda(p^*(a,b)-q)\int\sb{\mathbb{R}^N}\frac{g|u|^{q}}{|y|^{cp^*(a,c)}}\, dz.\label{bouchekif2.4C}
\end{align}

By the definition of the function $\psi_\lambda$, the subset $\mathcal{N}_\lambda$ can be written as the disjoint union of the three subsets
\[
  \begin{array}{r@{{}\equiv{}}l@{\qquad}r@{{}\equiv{}}l}
    \mathcal{N}_\lambda^+ 
    & \left\{u \in \mathcal{N}_\lambda\,| 
\, \langle \psi_\lambda'(u),u\rangle > 0\right\}, 
    & \mathcal{N}_\lambda^- 
    & \left\{u \in \mathcal{N}_\lambda\,| 
\, \langle \psi_\lambda'(u),u\rangle < 0\right\}, \\[\jot]
\multicolumn{4}{c}{\mathcal{N}_\lambda^0 \equiv
\left\{u \in \mathcal{N}_\lambda\,| 
\, \langle \psi_\lambda'(u),u\rangle = 0\right\}.}
  \end{array}
\]

Now we show that local minimizers in 
$\mathcal{N}_\lambda^+\cup \mathcal{N}_\lambda^-$ 
are critical points of $\varphi_\lambda$.
\begin{lema}\label{bouchekiflema2.4}
Suppose that there exists a local minimizer
$u_0\in\mathcal{N}_\lambda$ 
to the functional $\varphi_\lambda$ 
and that $u_0 \notin \mathcal{N}_\lambda^0$. 
Then $\varphi_\lambda'(u_0)=0$ 
in $\big(\mathcal{D}\sb{a}\sp{1,p}(\mathbb{R}\sp{N} \backslash \{|y|=0\})\big)^*$.
\end{lema}
\begin{proof}
If $u_0 \in \mathcal{N}_\lambda^+ \cup \mathcal{N}_\lambda^-$ is a local minimizer for the functional $\varphi_\lambda$, then $u_0$ solves the following problem: to minimize $\varphi_\lambda$ subjected to $\gamma(u)= 0$, 
where
$$
\gamma(u) = \int\sb{\mathbb{R}^N}\frac{|\nabla u|^{p}}{|y|^{ap}}\, dz
-\mu\int\sb{\mathbb{R}^N}\frac{|u|^p}{|y|^{p(a+1)}}\, dz
 -\int\sb{\mathbb{R}^N}\frac{h|u|^{p^*(a,b)}}{|y|^{bp^*(a,b)}}\, dz
-\lambda\int\sb{\mathbb{R}^N}\frac{g|u|^{q}}{|y|^{cp^*(a,c)}}\, dz
$$
is the constraint. Therefore, by the Lagrange's multipliers theory, there exists $\theta \in \mathbb{R}$ such that $\varphi_\lambda'(u_0)=\theta\gamma'(u_0)$. In this way,
\begin{align*}
\langle \varphi_\lambda'(u_0),v \rangle = \theta\langle\gamma'(u_0),v\rangle, 
\end{align*}
for every 
$v \in \mathcal{D}\sb{a}\sp{1,p}(\mathbb{R}\sp{N} \backslash \{|y|=0\})$. 
In particular, using $v=u_0$ and the fact that
$u_0 \in \mathcal{N}_\lambda^+ \cup \mathcal{N}_\lambda^-$, 
it follows that
\begin{align*}
0=\langle \varphi_\lambda'(u_0),u_0 \rangle = \theta\langle\gamma'(u_0),u_0\rangle.
\end{align*}

On the other hand,
\begin{align*}
\langle\gamma'(u_0),v\rangle
&= p\int\sb{\mathbb{R}^N}\frac{|\nabla u_0|^{p-2}\nabla u_0\nabla v}{|y|^{ap}}\, dz
-\mu p\int\sb{\mathbb{R}^N}\frac{|u_0|^{p-2}u_0v}{|y|^{p(a+1)}}\, dz\\
& \qquad -p^*(a,b)\int\sb{\mathbb{R}^N}\frac{h|u_0|^{p^*(a,b)-2}u_0v}{|y|^{bp^*(a,b)}}\, dz
-\lambda q\int\sb{\mathbb{R}^N}\frac{g|u_0|^{q-2}u_0v}{|y|^{cp^*(a,c)}}\, dz.
\end{align*} 
Using $v=u_0$ in the previous equation together with equation~\eqref{bouchekif2.4A}, we find
\begin{align*}
\langle\gamma'(u_0),u_0\rangle
&=\langle\psi'_\lambda(u_0),u_0\rangle \neq 0,
\end{align*}
where in the last passage we used the fact that
$u_0\notin \mathcal{N}_\lambda^0$. 
We deduce that $\theta = 0$ and, consequently, that 
$\varphi_\lambda'(u_0)=0$ in 
$\big(\mathcal{D}\sb{a}\sp{1,p}(\mathbb{R}\sp{N} \backslash \{|y|=0\})\big)^*$.
The lemma is proved.
\end{proof}

To apply Lemma~\ref{bouchekiflema2.4}, first we have to guarantee that a local minimizer $u_0 \in\mathcal{N}_\lambda$ does not belong to the subset $\mathcal{N}_\lambda^0$. This is the subject of the next result.

\begin{lema}\label{bouchekiflema2.5}
There exist $\Lambda_1 > 0$ such that for every 
$\lambda \in (0,\Lambda_1)$ we have $\mathcal{N}_\lambda^0 = \emptyset$.
\end{lema}
\begin{proof}
Suppose that $\mathcal{N}_\lambda^0\neq \emptyset$ for some 
$\lambda\in\mathbb{R}$ and let $u\in\mathcal{N}_\lambda^0$. 
By the definition of $\mathcal{N}_\lambda$ 
and by equations~\eqref{bouchekif2.4B} and~\eqref{bouchekif2.4C}, 
we obtain
\begin{align}\label{bouchekif2.6a}
\|u\|^p =\frac{p^*(a,b)-q}{p-q}\int\sb{\mathbb{R}^N}\frac{h|u|^{p^*(a,b)}}{|y|^{bp^*(a,b)}}\, dz
\quad \textrm{and} \quad
\|u\|^p = \lambda\frac{p^*(a,b)-q}{p^*(a,b)-p}\int\sb{\mathbb{R}^N}\frac{g|u|^{q}}{|y|^{cp^*(a,c)}}\, dz.
\end{align}

From the first equation in~\eqref{bouchekif2.6a}, 
using Maz'ya's inequality we determine
\begin{align*}
\left(\frac{p-q}{p^*(a,b)-q}\,\|h\|_{L^\infty}^{-1}K(N,p,\mu,a,b)^{-\frac{p^*(a,b)}{p}} \right)^\frac{1}{p^*(a,b)-p}
\leqslant \|u\|.
\end{align*}

On the other hand, from the the second equation in~\eqref{bouchekif2.6a}, using H\"{o}lder's and Maz'ya's inequalities we find
\begin{align*}
\|u\|
\leqslant \left(\lambda\frac{p^*(a,b)-q}{p^*(a,b)-p}\,\|g\|_{L\sb{\alpha}\sp{r}(\mathbb{R}\sp{N})}
K(N,p,\mu,a,c)^\frac{q}{p}\right)^\frac{1}{p-q}.
\end{align*}

Combining these two inequalities, we get
\begin{align*}
\lambda
& \geqslant (p^*(a,b)-p)(p^*(a,b)-q)^{-\frac{p^*(a,b)-q}{p^*(a,b)-p}}(p-q)^\frac{p-q}{p^*(a,b)-p}\|h\|_{L^\infty}^{-\frac{p-q}{p^*(a,b)-p}}\nonumber\\
& \qquad \times K(N,p,\mu,a,b)^{-\frac{p^*(a,b)(p-q)}{p(p^*(a,b)-p)}}\|g\|_{L\sb{\alpha}\sp{r}(\mathbb{R}\sp{N})}^{-1}K(N,p,\mu,a,c)^{-\frac{q}{p}} \equiv \Lambda_1.
\end{align*}
Hence, if 
$\lambda < \Lambda_1$, then 
$\mathcal{N}_\lambda^0 = \emptyset$. We remark that the value obtained for $\Lambda_1$ is not sharp.  This concludes the proof of the lemma.
\end{proof}

Using Lemma~\ref{bouchekiflema2.5}, for every $\lambda \in (0,\Lambda_1)$ we have
$\mathcal{N}_\lambda=\mathcal{N}_\lambda^+ \cup \mathcal{N}_\lambda^-$. 
Now we define the infima
\begin{align*}
d \equiv \inf_{u\in\mathcal{N}_\lambda} \varphi_\lambda(u),
\quad d^+ \equiv \inf_{u\in\mathcal{N}_\lambda^+} \varphi_\lambda(u),
\quad \text{and} 
\quad d^- \equiv \inf_{u\in\mathcal{N}_\lambda^-} \varphi_\lambda(u).
\end{align*}

\begin{lema}
\begin{enumerate}[label={\upshape(\arabic*)}, align=left, widest=ii, leftmargin=*]
\item\label{bouchekiflema2.6(i)}
If $\lambda \in (0,\Lambda_1)$, then $d\leqslant d^+ <0$.
\item\label{bouchekiflema2.6(ii)}
There exists a constant $C_1 = C_1(\lambda)> 0$ such that if $\lambda \in (0,(q/p)\Lambda_1)$, then $d^- > C_1>0$.
\end{enumerate}
\label{bouchekiflema2.6}
\end{lema}

\begin{proof}
(1) Let $u \in \mathcal{N}_\lambda^+$; 
by equations~\eqref{phineharih} and~\eqref{bouchekif2.4B}, 
and by the definition of $\mathcal{N}_\lambda^+$, it follows that
\begin{align*}
\varphi_\lambda(u)
& \leqslant \frac{(q-p)(p^*(a,b)-p)}{pqp^*(a,b)}\|u\|^p.
\end{align*}
Since $q<p$, we conclude that $d \leqslant d^+ <0$.

(2) Let $u \in \mathcal{N}_\lambda^-$; 
then, by equation~\eqref{bouchekif2.4B}, 
by the definition of $\mathcal{N}_\lambda^-$ 
and by Maz'ya's inequality, we find
\begin{align*}
\|u\| >  
\left(\frac{p-q}{p^*(a,b)-q}\|u\|_{L^\infty}^{-1}K(N,p,\mu,a,b)^{-\frac{p^*(a,b)}{p}}\right)^\frac{1}{p^*(a,b)-p}.
\end{align*}

Combining the previous inequality with inequality~\eqref{bouchekif2.3A}, we obtain
\begin{align*}
\varphi_\lambda(u)
& \geqslant \left(\frac{1}{p}-\frac{1}{p^*(a,b)}\right)(p^*(a,b)-1)^\frac{p}{p^*(a,b)-p}K(N,p,\mu,a,b)^\frac{p^*(a,b)}{p^*(a,b)-p}\\
& \qquad - \lambda\left(1-\frac{1}{p^*(a,b)}\right)(p^*(a,b)-1)^\frac{p}{p^*(a,b)-p}\|g\|_{L\sb{\alpha}\sp{r}(\mathbb{R}\sp{N})}
\equiv C_1(\lambda).
\end{align*}

On the other hand, direct computations show that  $C_1(\lambda)>0$ for $\lambda < (q/p) \Lambda_1.$
This concludes the proof of the lemma.
\end{proof}

Inspired by Tarantello~\cite{MR1168304} and by Hsu~\cite{MR2559275}, we prove the next results.

\begin{lema}\label{hsulema2.7.0}
Let $u\in \mathcal{D}\sb{a}\sp{1,p}(\mathbb{R}\sp{N} \backslash \{|y|=0\})$ be fixed and let the function 
$f_h\colon \mathbb{R}_*^+ \to \mathbb{R}$ be given by
\begin{align*}
f_h(t) = t^{p-q}\|u\|^p - t^{p^*(a,b)-q}\int\sb{\mathbb{R}^N}\frac{h|u|^{p^*(a,b)}}{|y|^{bp^*(a,b)}}\, dz
\end{align*}
The maximum of $f_h$ is such that
\begin{align}\label{hsu2.29}
f_h(t_{\max})
\geqslant \|u\|^q\left(\frac{p^*(a,b)-p}{p^*(a,b)-q}\right)
\left(\frac{p-q}{p^*(a,b)-q}\frac{1}{\|h\|_{L^\infty}K(N,p,\mu,a,b)^\frac{p^*(a,b)}{p}}\right)^\frac{p-q}{p^*(a,b)-p}.
\end{align}
\end{lema}
\begin{proof}
The proof follows directly from hypothesis~\eqref{hipoteseh1} and from Maz'ya's inequality.
\end{proof}
\begin{lema}\label{bouchekiflema2.7}
Let $\lambda \in (0,\Lambda_1)$ 
and let
$u\in \mathcal{D}\sb{a}\sp{1,p}(\mathbb{R}\sp{N} \backslash \{|y|=0\})$. 
Then the following assertions are valid.
\begin{enumerate}[label={\upshape(\arabic*)}, align=left, widest=ii, leftmargin=*]
\item \label{bouchekiflema2.7(i)}
If 
$\displaystyle\int\sb{\mathbb{R}^N}
\frac{g|u|^{q}}{|y|^{cp^*(a,c)}}\, dz \leqslant 0$, 
then there exists a unique number 
$t^- > t_{\max}$ such that 
$t^-u \in \mathcal{N}_\lambda^-$ and 
$\varphi_\lambda(t^-u) = \sup_{t\geqslant 0} \varphi_\lambda(tu)$.
\item \label{bouchekiflema2.7(ii)}
If 
$\displaystyle\int\sb{\mathbb{R}^N}
\frac{g|u|^{q}}{|y|^{cp^*(a,c)}}\, dz > 0$, 
then there exist unique numbers 
$0 < t^+ < t_{\max} < t^-$, such that 
$t^+u \in \mathcal{N}_\lambda^+$ and 
$t^-u \in \mathcal{N}_\lambda^-$; 
moreover,
$\varphi_\lambda(t^+u) = \inf_{0\leqslant t <t_{\max}} \varphi_\lambda(tu) $
and
$\varphi_\lambda(t^-u)=\sup_{t\geqslant 0} \varphi_\lambda(tu)$.
\end{enumerate}
\end{lema}

\begin{proof}
(1) If 
$\displaystyle\int\sb{\mathbb{R}^N}
\frac{g|u|^{q}}{|y|^{cp^*(a,c)}}\, dz \leqslant 0$ 
then there exists a number
$t^->t_{\max}$ such that $f_h(t^-)=\lambda\displaystyle\int\sb{\mathbb{R}^N}\frac{g|u|^{q}}{|y|^{cp^*(a,c)}}\, dz$ and $f'_h(t^-) <0$; moreover, this number is unique. 

Now we show that $t^-u \in \mathcal{N}_\lambda$. 
Indeed,
\begin{align*}
\langle \varphi_\lambda'(t^-u),t^-u\rangle
&= (t^-)^p\|u\|^p - (t^-)^{p^*(a,b)}\int\sb{\mathbb{R}^N}\frac{h|u|^{p^*(a,b)}}{|y|^{bp^*(a,b)}}\, dz
-\lambda (t^-)^q\int\sb{\mathbb{R}^N}\frac{g|u|^{q}}{|y|^{cp^*(a,c)}}\, dz\nonumber\\
&=(t^-)^p\left(f_h(t^-)-\lambda \int\sb{\mathbb{R}^N}\frac{g|u|^{q}}{|y|^{cp^*(a,c)}}\, dz \right)=0.
\end{align*}
And to show that $t^-u \in \mathcal{N}_\lambda^-$, we use the equation~\eqref{bouchekif2.4B} to get
\begin{align*}
&\langle \psi_\lambda'(t^-u),t^-u\rangle
=(t^-)^{q+1}f_h'(t^-)<0.
\end{align*}

Finally, by the definitions of $\varphi_\lambda$ and $f_h$, we deduce that
\begin{align*}
\frac{d}{dt}\varphi_\lambda(t\sp{-}u)
&=(t\sp{-})^{q-1}\left(f_h(t\sp{-}) -\lambda \int\sb{\mathbb{R}^N}\frac{g|u|^{q}}{|y|^{cp^*(a,c)}}\, dz \right) = 0.
\end{align*}
And since
\begin{align*}
f_h'(t)
&=(p-q)t^{p-q-1}\|u\|^p - (p^*(a,b)-q)t^{p^*(a,b)-q-1}\int\sb{\mathbb{R}^N}\frac{h|u|^{p^*(a,b)}}{|y|^{bp^*(a,b)}}\, dz,
\end{align*}
for \( 0 < t < t_{\max} \) we have \( f_h'(t) > 0 \)
and for \( t>t_{\max} \) we have \( f_h'(t) < 0 \). 
Hence, for \( t^- > t_{\max} \) it follows that
\begin{align*}
\frac{d^2}{dt^2}\varphi_\lambda(t^-u)
&=\langle\psi_\lambda'(t^-u),t^-u\rangle\\ 
&= (p-q)\|t^-u\|^p - (p^*(a,b)-q)\int\sb{\mathbb{R}^N}\frac{h|t^-u|^{p^*(a,b)}}{|y|^{bp^*(a,b)}}\, dz
= (t^-)^{q+1}f_h'(t) < 0.
\end{align*}
This implies that $\varphi_\lambda(t^-u) = \sup_{t\geqslant 0}\varphi_\lambda(tu)$.

(2) If 
$\displaystyle\int\sb{\mathbb{R}^N}
\frac{g|u|^{q}}{|y|^{cp^*(a,c)}}\, dz > 0$, 
then using H\"{o}lder's and Maz'ya's inequalities, we deduce that
\begin{align*}
0 = f_h(0) &< \lambda\int\sb{\mathbb{R}^N}\frac{g|u|^{q}}{|y|^{cp^*(a,c)}}\, dz\nonumber\\
& < \Lambda_1\|g\|_{L\sb{\alpha}\sp{r}(\mathbb{R}\sp{N})}K(N,p,\mu,a,c)^\frac{q}{p}\|u\|^q\nonumber\\
&=\|u\|^q\left(\frac{p^*(a,b)-p}{p^*(a,b)-q}\right)
\left(\frac{p-q}{p^*(a,b)-q}\frac{1}{\|h\|_{L^\infty}K(N,p,\mu,a,b)^\frac{p^*(a,b)}{p}}\right)^\frac{p-q}{p^*(a,b)-p}\nonumber\\
&\leqslant f_h(t_{\max})
\end{align*}
for every $\lambda\in(0,\Lambda_1)$, where we have used inequality~\eqref{hsu2.29} in the last passage. So, by the previous inequality there exist numbers $0<t^+<t_{max}<t^-$ such that
\begin{align*}
&f_h(t^+) = \lambda\int\sb{\mathbb{R}^N}\frac{g|u|^{q}}{|y|^{cp^*(a,c)}}\, dz = f_h(t^-) \quad \textrm{and} \quad
f_h'(t^+) > 0 > f_h'(t^-);
\end{align*}
moreover, these numbers are unique.
Computations similar to those in the proof of item~\ref{bouchekiflema2.7(i)} show that 
$t^+u \in \mathcal{N}_\lambda^+$, and 
$t^-u \in \mathcal{N}_\lambda^-$,
as well as 
$\varphi_\lambda(t^-u) \geqslant \varphi_\lambda(tu) \geqslant \varphi_\lambda(t^+u)$ 
for every $t\in[t^+,t^-]$ 
and also 
$\varphi_\lambda(t^+u) \leqslant \varphi_\lambda(tu)$ 
for every $t \in [0,t^+]$. 
In this way,
\begin{align*}
\varphi_\lambda(t^+u) = \inf_{0\leqslant t\leqslant t_{\max}} \varphi_\lambda(tu)
\quad \text{and} \quad \varphi_\lambda(t^-u) = \sup_{t \geqslant 0} \varphi_\lambda(tu).
\end{align*}
This completes the proof of the lemma.
\end{proof}

To close this section we state one more result which will be useful in the proof of the existence of Palais-Smales sequences for the functional 
$\varphi_\lambda$ at the levels $d$ and $d^-$.
\begin{lema}\label{hsulema3.1}
If $\lambda \in (0,\Lambda_1)$, then for every 
$u\in \mathcal{N}_\lambda$ 
there exist $\epsilon > 0$ and a differentiable function 
$\xi \colon B_\epsilon(0) \subset \mathcal{D}\sb{a}\sp{1,p}(\mathbb{R}\sp{N} \backslash \{|y|=0\}) \to \mathbb{R}^+$ 
with $\xi(0)=1$ and such that $\xi(\bar{v})(u-\bar{v}) \in \mathcal{N}_\lambda$  and
\begin{align}\label{hsu3.1}
\langle\xi'(0),v \rangle
&= \frac{\left\{\splitfrac{p \displaystyle\int\sb{\mathbb{R}^N}
\frac{|\nabla u|^{p-2}\nabla u \nabla v}{|y|^{ap}}\, dz
- p \mu \displaystyle\int\sb{\mathbb{R}^N}
\frac{|u|^{p-2}uv}{|y|^{p(a+1)}}\, dz }
{-p^*(a,b) \displaystyle\int\sb{\mathbb{R}^N}
\frac{h|u|^{p^*(a,b)-2}uv}{|y|^{bp^*(a,b)}}\, dz
-\lambda q \displaystyle\int\sb{\mathbb{R}^N}\frac{g|u|^{q-2}uv}{|y|^{cp^*(a,c)}}\, dz}\right\}}
{(p-q)\|u\|^p-(p^*(a,b)-q) \displaystyle\int\sb{\mathbb{R}^N}\frac{h|u|^{p^*(a,b)}}{|y|^{bp^*(a,b)}}\, dz}
\end{align}
for every $v \in \mathcal{D}\sb{a}\sp{1,p}(\mathbb{R}\sp{N} \backslash \{|y|=0\})$. In particular, if 
$u \in \mathcal{N}_\lambda^-$ then $\xi^-(\bar{v})(u-\bar{v}) \in \mathcal{N}_\lambda^-$.
\end{lema}

\begin{proof}
Given $u \in \mathcal{N}_\lambda$, we define the function 
$F_u\colon \mathbb{R}\times \mathcal{D}\sb{a}\sp{1,p}(\mathbb{R}\sp{N} \backslash \{|y|=0\}) \to \mathbb{R}$ by
\begin{align*}
F_u(\xi,w) 
&\equiv \langle \varphi_\lambda'(\xi(u-w),\xi(u-w)\rangle\nonumber\\
&= \xi^p\|u-w\|^p -\xi^{p^*(a,b)}\int\sb{\mathbb{R}^N}\frac{h|u-w|^{p^*(a,b)}}{|y|^{bp^*(a,b)}}\, dz
-\lambda \xi^q\int\sb{\mathbb{R}^N}\frac{g|u-w|^{q}}{|y|^{cp^*(a,c)}}\, dz.
\end{align*}
Then $F_u(1,0) = \langle\varphi_\lambda'(u),u\rangle = 0$ and by 
equations~\eqref{nehari1} and~\eqref{bouchekif2.4B}, as well as the fact that $\mathcal{N}_\lambda^0=\emptyset$, it follows that
\begin{align*}
\frac{d}{d\xi}F_u(1,0)
= (p-q)\|u\|^p-(p^*(a,b)-q)\int\sb{\mathbb{R}^N}\frac{h|u|^{p^*(a,b)}}{|y|^{bp^*(a,b)}}\, dz \neq 0.
\end{align*}

By the implicit function theorem there exist 
$\epsilon \in \mathbb{R}^+$ and 
a differentiable function
$\xi\colon B_\epsilon(0) 
\subset \mathcal{D}\sb{a}\sp{1,p}(\mathbb{R}\sp{N} \backslash \{|y|=0\}) 
\to \mathbb{R}$ such that $\xi(0)=1$ and
$F_u(\xi(\bar{v}),\bar{v})=0$. This is equivalent to
$$
\langle \varphi_\lambda'(\xi(\bar{v})(u-\bar{v})),\xi(\bar{v})(u-\bar{v})\rangle =0,
$$
that is, $\xi(\bar{v})(u-\bar{v}) \in \mathcal{N}_\lambda$ 
for every 
$\bar{v} \in B_\epsilon(0) \subset \mathcal{D}\sb{a}\sp{1,p}(\mathbb{R}\sp{N} 
\backslash \{|y|=0\})$.
The proof of the last claim of the lemma follows from this case with some minors modifications. 

To conclude, the derivative of $\xi$ at the origin applied to the function 
\( v \in \mathcal{D}\sb{a}\sp{1,p}(\mathbb{R}\sp{N} 
\backslash \{|y|=0\}) \) 
is given by
\begin{align*}
\langle\xi'(0),v\rangle
&=-\displaystyle\frac{\partial F_u(\xi(0),v)}{\partial w}\left(\displaystyle\frac{\partial F_u(\xi(0),v)}{\partial\xi}\right)^{-1}
\end{align*}
which implies~\eqref{hsu3.1}.
The lemma is proved.
\end{proof}

\section{Proof of Theorem~\ref{teoremamultiplicidade1}}
\label{multiplicidadesecao2}

In this section we prove Theorem~\ref{teoremamultiplicidade1} by showing 
that there exists a Palais-Smale sequence at an appropriate negative level 
and also by showing that we can recover the compactness condition to such 
sequences.

\begin{prop}
\begin{enumerate}[label={\upshape(\arabic*)}, align=left, widest=ii, leftmargin=*]
\item\label{hsuprop3.3(i)}
For every $\lambda\in(0,\Lambda_1)$ there exists a Palais-Smale sequence
$(u_n)_{n \in \mathbb{N}}\subset \mathcal{N}_\lambda$
for the functional $\varphi_\lambda$ at the level $d$; 
that is,
\( \varphi_\lambda(u_n) \to d \)
and
\( \varphi_\lambda'(u_n) \to 0 \)
in the dual space
\( \left(\mathcal{D}\sb{a}\sp{1,p}(\mathbb{R}\sp{N} 
\backslash \{|y|=0\})\right)^* \)
as \( n \to \infty \).

\item\label{hsuprop3.3(ii)}
For every $\lambda\in(0,(q/p)\Lambda_1)$ there exists a Palais-Smale sequence 
$(u_n)_{n \in \mathbb{N}}\subset \mathcal{N}_\lambda^-$ 
for the functional $\varphi_\lambda$
at the level $d^-$; 
that is,
\( \varphi_\lambda(u_n) \to d^- \)
and
\( \varphi_\lambda'(u_n) \to 0 \)
in the dual space
\( \left(\mathcal{D}\sb{a}\sp{1,p}(\mathbb{R}\sp{N} 
\backslash \{|y|=0\})\right)^* \)
as \( n \to \infty \).

\end{enumerate}
\label{hsuprop3.3}
\end{prop}
\begin{proof}
(1) By Lemma~\ref{bouchekiflema2.3} we know that the functional
$\varphi_\lambda$ is coercive and bounded from below in the Nehari manifold
$\mathcal{N}_\lambda$. 
By Ekeland's variational principle there exists a minimizing sequence
$(u_n)_{n\in\mathbb{N}} \subset \mathcal{N}_\lambda$ 
such that
\begin{align*}
\varphi_\lambda(u_n) < d+\frac{1}{n}
\quad \text{and} \quad
\varphi_\lambda(u_n) < \varphi_\lambda(w) +\frac{1}{n}\|w-u_n\| 
\textrm{ for every } w\in\mathcal{N}_\lambda.
\end{align*}
Since $d<0$, considering 
$n\in \mathbb{N}$ big enough we get
\begin{align*}
\varphi_\lambda(u_n) 
&<d+\frac{1}{n} <\frac{d}{p} <0.
\end{align*}

Combining this inequality with equation~\eqref{bouchekif2.3A},
as well as with the definition of the functional $\varphi_\lambda$ and 
H\"{o}lder's inequality, we obtain
\begin{align*}
&\left(\frac{1}{p}-\frac{1}{p^*(a,b)}\right) \|u_n\|^p
-\lambda\left(\frac{1}{q}-\frac{1}{p^*(a,b)}\right)\|g\|_{L\sb{\alpha}\sp{r}(\mathbb{R}\sp{N})}K(N,p,\mu,a,c)^\frac{q}{p}\|u_n\|^q 
< \frac{d}{p} < 0.
\end{align*}
This implies that 
\begin{align}\label{hsu3.12b}
\|u_n\|< \left(\left(\frac{1}{q}-\frac{1}{p^*(a,b)}\right)\lambda\|g\|_{L\sb{\alpha}\sp{r}(\mathbb{R}\sp{N})}K(N,p,\mu,a,c)^\frac{q}{p}\left(\frac{1}{p}-\frac{1}{p^*(a,b)}\right)^{-1}\right)^\frac{1}{p-q} 
\end{align}
and 
\begin{align*}
0 < -\frac{d}{p}\left(\frac{1}{q}-\frac{1}{p^*(a,b)}\right)^{-1} 
< \lambda\int\sb{\mathbb{R}^N}\frac{g|u_n|^{q}}{|y|^{cp^*(a,c)}}\, dz
\leq \lambda\|g\|_{L\sb{\alpha}\sp{r}(\mathbb{R}\sp{N})}K(N,p,\mu,a,c)^\frac{q}{p}\|u_n\|^q. 
\end{align*}
Consequently, $u_n \neq 0$ and
\begin{align}\label{hsu3.12a}
\|u_n\| > \left(-\frac{d}{p} 
\left(\frac{1}{q}-\frac{1}{p^*(a,b)}\right)^{-1}
\frac{1}{\lambda\|g\|_{L\sb{\alpha}\sp{r}(\mathbb{R}\sp{N})}K(N,p,\mu,a,c)^\frac{q}{p}}
\right)^\frac{1}{q}.
\end{align}

Now we show that
$
\varphi_\lambda'(u_n) \to 0
$
in$\left(\mathcal{D}\sb{a}\sp{1,p}(\mathbb{R}\sp{N} \backslash \{|y|=0\})\right)^*$ as $n\to \infty$. 
To do this, we apply Lemma~\ref{hsulema3.1} to the elements of the minimizing sequence 
$(u_n)_{n \in \mathbb{N}} \subset \mathcal{N}_\lambda$. Hence, for every
$u_n\in\mathcal{N}_\lambda$ there exist $\epsilon_n > 0$ 
and a differentiable function
$\xi_n\colon B_{\epsilon_n}(0) \to \mathbb{R}$ 
such that
$\xi_n(w)(u_n-w) \in \mathcal{N}_\lambda$.

Now we consider $0<\rho<\epsilon_n$. 
For $u\in \mathcal{D}\sb{a}\sp{1,p}(\mathbb{R}\sp{N} \backslash \{|y|=0\})$ 
such that $u\neq0$, we define
$w_\rho \equiv \rho u/\|u\|$ 
and also 
$\eta_\rho \equiv \xi_n(w_\rho)(u_n-w_\rho)$. 
By the definition of 
$\xi_n(w_\rho)$ it follows that
$\eta_\rho \in \mathcal{N}_\lambda$, which implies that
\begin{align}\label{hsu3.13.2}
\langle\varphi_\lambda'(\eta_\rho),\eta_\rho\rangle = \langle\varphi_\lambda'(\eta_\rho),\xi_n(w_\rho)(u_n-w_\rho)\rangle=0.
\end{align} 
We also deduce from the properties of the minimizing sequence that
$\varphi_\lambda(\eta_\rho)-\varphi_\lambda(u_n) \geqslant -(1/n)\|\eta_\rho - u_n\|. 
$
By the mean value theorem, we conclude that
\begin{align*}
\langle \varphi_\lambda'(u_n), \eta_\rho - u_n\rangle + o(\|\eta_\rho - u_n\|)
\geqslant -\frac{1}{n}\|\eta_\rho - u_n\|.
\end{align*}
From this, we get
\begin{align}\label{hsu3.16}
\langle\varphi_\lambda'(u_n),-w_\rho\rangle + (\xi_n(w_\rho)-1)\langle\varphi_\lambda'(u_n), (u_n-w_\rho) \rangle 
\geqslant -\frac{1}{n}\|\eta_\rho - u_n\| + o(\|\eta_\rho - u_n\|).
\end{align}
Multiplying both sides of equation~\eqref{hsu3.13.2} by the factor
$(\xi_n(w_\rho)-1)/\xi_n(w_\rho)$
and adding termwise with inequality~\eqref{hsu3.16}, we obtain
\begin{align*}
&\langle \varphi_\lambda'(u_n),- \rho \dfrac{u}{\| u \|} \rangle
+ (\xi_n(w_\rho)-1)\langle\varphi_\lambda'(u_n),u_n-w_\rho\rangle
+ \frac{(\xi_n(w_\rho)-1)}{\xi_n(w_\rho)}\langle\varphi_\lambda'(\eta_\rho),\xi_n(w_\rho)(u_n-w_\rho)\rangle\nonumber\\
&\qquad \qquad =-\rho\langle \varphi_\lambda'(u_n), \frac{u}{\|u\|}\rangle
+ (\xi_n(w_\rho)-1)\langle\varphi_\lambda'(u_n)-\varphi_\lambda'(\eta_\rho),u_n-w_\rho\rangle\nonumber\\
& \qquad \qquad \geqslant -\frac{1}{n}\|\eta_\rho - u_n\| + o(\|\eta_\rho - u_n\|).
\end{align*}
Therefore, by the definition of \( w\sb{\rho} \) we get
\begin{align}\label{hsu3.18}
&\langle \varphi_\lambda'(u_n), \frac{u}{\|u\|}\rangle \leqslant \frac{(\xi_n(w_\rho)-1)}{\rho}\langle\varphi_\lambda'(u_n)-\varphi_\lambda'(\eta_\rho),u_n-w_\rho\rangle
+\frac{1}{\rho n}\|\eta_\rho - u_n\| + \frac{o(\|\eta_\rho - u_n\|)}{\rho}.
\end{align}

Following up, we have
\begin{align*}
\|\eta_\rho - u_n\|
&\leqslant\rho\xi_n(w_\rho) + (\xi_n(w_\rho)-1)\|u_n\|,
\end{align*}
as well as
\begin{align*}
\lim_{\rho \to 0}\frac{|\xi_n(w_\rho)-1|}{\rho}
& = \lim_{\rho \to 0}\frac{|\xi_n(w_\rho)-\xi_n(0)|}{\rho} \leqslant\|\xi_n'(0)\|.
\end{align*}

Hereafter, let $C$ stands for a positive constant that can change from one 
passage to another. 
Keeping $n\in \mathbb{N}$ fixed, 
using inequalities~\eqref{hsu3.12b} and~\eqref{hsu3.12a}, 
passing to the limit as  
$\rho \to 0$ in inequality~\eqref{hsu3.18}, and 
noticing that $\lim_{\rho \to 0} \eta_\rho = u_n$, it follows that
\begin{align}\label{hsu3.20}
\langle\varphi_\lambda'(u_n),\frac{u}{\|u\|}\rangle
& \leqslant 
\lim\sb{\rho \to 0}
\left\{
\splitfrac{\dfrac{1}{\rho}|\xi_n(w_\rho)-1|\langle \varphi_\lambda'(u_n)-\varphi_\lambda'(\eta_\rho),u_n-w_\rho\rangle }
{ +\dfrac{\xi_n(w_\rho)}{n}
  +\dfrac{(\xi_n(w_\rho)-1)}{\rho}
   \dfrac{\|u_n\|}{n}
  +\dfrac{o(\rho\xi_n(w_\rho)+|\xi_n(w_\rho)-1|\|u_n\|)}{\rho} }
\right\} \nonumber \\
&< \frac{C}{n} \left( 1+ \|\xi_n'(0)\| \right).
\end{align}

Now we show that the sequence 
$\left(\|\xi_n'(u)\|\right)_{n \in \mathbb{N}} \subset \mathbb{R}$
is uniformly bounded. By equation~\eqref{hsu3.1}, 
and using inequalities~\eqref{hsu3.12b} 
and~\eqref{hsu3.12a},  
together with H\"{o}lder's and Maz'ya's inequalities, 
we deduce that
\begin{align}\label{hsu3.21}
\langle\xi_n'(0),v\rangle
\leqslant \frac{C \|v\|}{\left|(p-q)\|u_n\|^p-(p^*(a,b)-q)\displaystyle\int\sb{\mathbb{R}^N}\frac{h|u_n|^{p^*(a,b)}}{|y|^{bp^*(a,b)}}\, dz\right|}
\end{align}
for some positive constant $C > 0$ which does not depend on $n \in \mathbb{N}$.

\begin{afirmativa}
For $n \in \mathbb{N}$ big enough the denominator of
inequality~\eqref{hsu3.21} is bounded.
\end{afirmativa}
\begin{proof}[Proof of the claim]
To prove this claim we argue by contradiction and we suppose that there exists a subsequence $(u_n)_{n\in\mathbb{N}} \in \mathcal{N}_\lambda$ such that
\begin{align*}
(p-q)\|u_n\|^p-(p^*(a,b)-q)\int\sb{\mathbb{R}^N}\frac{h|u_n|^{p^*(a,b)}}{|y|^{bp^*(a,b)}}\, dz = o(1).
\end{align*}
This relation, together with the fact that $(u_n)_{n\in\mathbb{N}}\in \mathcal{N}_\lambda$, imply that
\begin{align*}
\|u_n\|^p = \lambda\frac{p^*(a,b)-q}{p^*(a,b)-p}\int\sb{\mathbb{R}^N}\frac{g|u_n|^{q}}{|y|^{cp^*(a,c)}}\, dz + o(1).
\end{align*}

Additionaly, together with 
H\"{o}lder's and Maz'ya's inequalities and the previous relations, we deduce that
\begin{align*}
\|u_n\| \geqslant \left(\frac{p-q}{p^*(a,b)-p}\frac{1}{\|h\|_{L^\infty}}K(N,p,\mu,a,b)^{-\frac{p^*(a,b)}{p}}\right)^\frac{1}{p^*(a,b)-p}+o(1).
\end{align*}
and
\begin{align*}
\|u_n\|
\leqslant \left(\lambda\frac{p^*(a,b)-q}{p^*(a,b)-p}\|g\|_{L\sb{\alpha}\sp{r}(\mathbb{R}\sp{N})}
K(N,p,\mu,a,c)^\frac{q}{p}\right)^\frac{1}{p-q}+o(1).
\end{align*}

Combining these last two inequalities, we get
\begin{align*}
\lambda
& \geqslant (p^*(a,b)-p)(p^*(a,b)-q)^{-\frac{p^*(a,b)-q}{p^*(a,b)-p}}(p-q)^\frac{p-q}{p^*(a,b)-p}\|h\|_{L^\infty}^{-\frac{p-q}{p^*(a,b)-p}}\\
& \qquad \times K(N,p,\mu,a,b)^{-\frac{p^*(a,b)(p-q)}{p(p^*(a,b)-p)}}\|g\|_{L\sb{\alpha}\sp{r}(\mathbb{R}\sp{N})}^{-1}K(N,p,\mu,a,c)^{-\frac{q}{p}}
= \Lambda_1,
\end{align*}
which is a contradiction. Hence, our claim follows. 
\end{proof}

In this way, by inequality~\eqref{hsu3.21} we have
$\langle\xi_n'(0),v\rangle \leqslant C\|v\|$, that is,
$\|\xi_n'(0)\| \leqslant C$; and from inequality~\eqref{hsu3.20} we deduce that 
$
\langle\varphi_\lambda'(u_n),u/\|u\|\rangle
 \to 0
$
as $n \to \infty$. 
Finally, $\|\varphi_\lambda'(u_n)\| \to 0$ in 
$\left(\mathcal{D}\sb{a}\sp{1,p}(\mathbb{R}\sp{N} 
\backslash \{|y|=0\})\right)^*$ as $n\to \infty$, 
which concludes the proof of item~\ref{hsuprop3.3(i)}.

(2) The proof of this item is similar to that of item~\ref{hsuprop3.3(i)}; one just have to use the last claim of Lemma~\ref{hsulema3.1}.
\end{proof}

Now we connect these auxiliary results to obtain a minimum for the functional $\varphi_\lambda$ in $\mathcal{N}_\lambda^+$.

\begin{prop}\label{bouchekifprop3.2}
If $\lambda \in (0,\Lambda_1)$, then the functional $\varphi_\lambda$ 
has a positive minimizer $u_1 \in \mathcal{N}_\lambda^+$ such that
$\varphi_\lambda(u_1)=d=d^+<0$
and the function $u_1$ is solution to problem~\eqref{problemamultiplicidade} in $\mathcal{D}\sb{a}\sp{1,p}(\mathbb{R}\sp{N} \backslash \{|y|=0\})$.
\end{prop}

\begin{proof}
By Lemma~\ref{bouchekiflema2.3} the functional
$\varphi_\lambda$ is coercive and bounded from below in the 
Nehari manifold $\mathcal{N}_\lambda$. We can suppose, up to a subsequence still denoted in the same way, that there exists
$u_1 \in \mathcal{D}\sb{a}\sp{1,p}(\mathbb{R}\sp{N} \backslash \{|y|=0\})$ 
such that, as \( n \to \infty \), 
\begin{enumerate}[label={\upshape(\arabic*)}, align=left, widest=4, leftmargin=*]
\item
$u_n \rightharpoonup u_1$ weakly in 
$\mathcal{D}\sb{a}\sp{1,p}(\mathbb{R}\sp{N} \backslash \{|y|=0\})$;
\item
$u_n \rightharpoonup u_1$ weakly in $L_b^{p^*(a,b)}(\mathbb{R}^N\backslash\{|y|=0\})$;
\item
$u_n \rightharpoonup u_1$ weakly in $L_c^{q}(\mathbb{R}^N\backslash\{|y|=0\})$, 
for every $q$ such that $1<q<p$ and $a\leqslant c <a+1$;
\item
$u_n \to u_1$ a.\@ e.\@ in $\mathbb{R}^N$.
\end{enumerate}
Hence, by Proposition~\ref{hsuprop3.3}~\ref{hsuprop3.3(i)} and by the previous convergences it follows that $u_1$ is a weak solution to problem~\eqref{problemamultiplicidade}. Indeed, we have
\begin{align*}
&\lim_{n\to \infty}
\left\{
\splitfrac{\displaystyle\int\sb{\mathbb{R}^N}\frac{|\nabla u_n|^{p-2}\nabla u_n \nabla v}{|y|^{ap}}\, dz
-\mu \displaystyle\int\sb{\mathbb{R}^N}
\frac{|u_n|^{p-2}u_nv}{|y|^{p(a+1)}}\, dz }
{ -\displaystyle\int\sb{\mathbb{R}^N}
 \frac{h|u_n|^{p^*(a,b)-2}u_nv}{|y|^{bp^*(a,b)}}\, dz
-\lambda \displaystyle\int\sb{\mathbb{R}^N}
\frac{g|u_n|^{q-2}u_nv}{|y|^{cp^*(a,c)}}\, dz }
\right\} \\
& \qquad = \int\sb{\mathbb{R}^N}\frac{|\nabla u_1|^{p-2}\nabla u_1 \nabla v}{|y|^{ap}}\, dz
-\mu\int\sb{\mathbb{R}^N}\frac{|u_1|^{p-2}u_1v}{|y|^{p(a+1)}}\, dz\\
&\qquad \qquad -\int\sb{\mathbb{R}^N}\frac{h|u_1|^{p^*(a,b)-2}u_1v}{|y|^{bp^*(a,b)}}\, dz
-\lambda \int\sb{\mathbb{R}^N}\frac{g|u_1|^{q-2}u_1v}{|y|^{cp^*(a,c)}}\, dz
\end{align*}
for every
$v \in \mathcal{D}\sb{a}\sp{1,p}(\mathbb{R}\sp{N} \backslash \{|y|=0\})$. 
Moreover, by Lemma~\ref{bouchekiflema2.6}~\ref{bouchekiflema2.6(i)} 
we have $\varphi_\lambda(u_1)=d<0=\varphi_\lambda(0)$, 
that is, $u_1\neq 0$.

Now we show that  $u_n \to u_1$ strongly in
$\mathcal{D}\sb{a}\sp{1,p}(\mathbb{R}\sp{N} \backslash \{|y|=0\})$. We argue by contradiction and we suppose that the inequality
$\|u_1\| < \liminf_{n\to\infty} \|u_n\|$ is valid. 
Hence, by equation~\eqref{phineharig} we have
\begin{align*}
d \leqslant \varphi_\lambda(u_1) 
& < \liminf_{n\to\infty} \left(\frac{1}{p}-\frac{1}{p^*(a,b)}\right)\|u_n\|^p 
-\lambda\left(\frac{1}{q}-\frac{1}{p^*(a,b)}\right)\int\sb{\mathbb{R}^N}\frac{g|u_n|^{q}}{|y|^{cp^*(a,c)}}\, dz\\
&=\liminf_{n\to\infty} \varphi_\lambda (u_n) = d,
\end{align*} 
which is a contradiction. Therefore, 
$u_n \to u_1$ strongly in  
$\mathcal{D}\sb{a}\sp{1,p}(\mathbb{R}\sp{N} \backslash \{|y|=0\})$.

It remains to show that 
$u_1 \in \mathcal{N}_\lambda^+$ and that this function is positive. Again we argue by contradiction and we suppose that 
$u_1 \notin \mathcal{N}_\lambda^+$. 
Knowing that $\mathcal{N}^0 = \emptyset$, from Lemma~\ref{bouchekiflema2.5}, we deduce that 
$u_1 \in \mathcal{N}_\lambda^-$. 
And since 
$d<0$, by equation~\eqref{phineharig} we have
\begin{align*}
d = \varphi_\lambda(u_1) 
& =\left(\frac{1}{p}-\frac{1}{p^*(a,b)}\right)\|u_1\|^p -\lambda\left(\frac{1}{q}-\frac{1}{p^*(a,b)}\right)\int\sb{\mathbb{R}^N}\frac{g|u_1|^{q}}{|y|^{cp^*(a,c)}}\, dz < 0;
\end{align*}
consequently, 
$\displaystyle\int\sb{\mathbb{R}^N}\frac{g|u_1|^{q}}{|y|^{cp^*(a,c)}}\, dz >0$. 
Applying  Lemma~\ref{bouchekiflema2.7}~\ref{bouchekiflema2.7(ii)}, 
we deduce that there exist numbers 
$0<t_1^+ < t_{\max} < t_1^-$ such that 
$t_1^+u_1 \in \mathcal{N}_\lambda^+$
and 
$t_1^-u_1 \in \mathcal{N}_\lambda^-$; 
moreover,
\begin{align*}
\varphi_\lambda(t_1^+u_1) = \inf_{0<t\leqslant t_{\max}} \varphi_\lambda(tu_1) \quad \textrm{and} \quad
\varphi_\lambda(t_1^-u_1) = \sup_{t \geqslant 0} \varphi_\lambda(tu_1).
\end{align*}
In particular, $t_1^-u_1=u_1$, that is, $t_1^- = 1$, 
because $u_1 \in \mathcal{N}_\lambda^-$ and the numbers
$t_1^-$ and $t_1^+$ are unique. As we have already seen in the proof of Lemma~\ref{bouchekiflema2.7}~\ref{bouchekiflema2.7(ii)}, we obtain the relations 
\begin{align*}
\frac{d}{dt}\varphi_\lambda(t_1^+u_1)=0 
\quad \textrm{and} \quad
\frac{d^2}{dt^2}\varphi_\lambda(t_1^+u_1)>0.
\end{align*}
Hence, there exists $t^- \in \mathbb{R}$ 
such that $t_1^+<t^-\leqslant t_1^- = 1$ 
and with
$\varphi_\lambda(t_1^+u_1) < \varphi_\lambda(t_1^-u_1)$. 
Again by Lemma~\ref{bouchekiflema2.7}~\ref{bouchekiflema2.7(ii)} 
we deduce that
\begin{align*}
\inf_{0 \leqslant t < t_{\max}}\varphi_\lambda(tu)=
\varphi_\lambda(t_1^+u_1)< \varphi_\lambda(t^-u_1)<\varphi_\lambda(t_1^-u_1)
= \sup_{t\geqslant 0 }\varphi_\lambda(tu)
=\varphi_\lambda(u_1)=d,
\end{align*}
which is a contradiction; 
therefore, $u\in\mathcal{N}_\lambda^+$. 
To guarantee the positivity of the solution, it is sufficient to note that 
$\varphi_\lambda(u_1) = \varphi_\lambda(|u_1|)$ 
and that $|u_1| \in \mathcal{N}_\lambda^+$. 
This concludes the proof of the proposition.
\end{proof}

\begin{proof}[Proof of Theorem~\ref{teoremamultiplicidade1}]
The result follows immediately from Proposition~\ref{bouchekifprop3.2}.
\end{proof}

\section{Proof of Theorem~\ref{teoremamultiplicidade2}}
\label{multiplicidadesecao3}

In this section we establish the existence of a second solution to problem~\eqref{problemamultiplicidade}. To do this, we state some more auxiliary results.

\begin{lema}\label{bouchekiflema4.1}
Suppose that hypotheses~\eqref{hipoteseg1} and~\eqref{hipoteseg2} are valid. 
Let $(u_n)_{n \in \mathbb{N}} \in 
\mathcal{D}\sb{a}\sp{1,p}(\mathbb{R}\sp{N} \backslash \{|y|=0\})$ 
be a Palais-Smale sequence for the functional
$\varphi_\lambda$ at some level $d \in \mathbb{R}$ and suppose that 
$u_n \rightharpoonup u$ weakly in 
$\mathcal{D}\sb{a}\sp{1,p}(\mathbb{R}\sp{N} \backslash \{|y|=0\})$. 
Then $\varphi_\lambda'(u)=0$ and
$\varphi_\lambda(u) \geqslant - C_0\lambda^\frac{p}{p-q}$.
\end{lema}
\begin{proof}
Since $(u_n)_{n\in\mathbb{N}} \subset \mathcal{D}\sb{a}\sp{1,p}(\mathbb{R}\sp{N} \backslash \{|y|=0\})$ is a Palais-Smale sequence and 
$u_n \rightharpoonup u$ as \( n \to \infty\), it follows that
\begin{align*}
0  
= \lim_{n\to\infty} \langle\varphi_\lambda'(u_n),v\rangle
&= \lim_{n\to\infty} \Bigg(\int\sb{\mathbb{R}^N}\frac{|\nabla u_n|^{p-2}\nabla u_n \nabla v}{|y|^{ap}}\, dz
-\mu\int\sb{\mathbb{R}^N}\frac{|u_n|^{p-2}u_nv}{|y|^{p(a+1)}}\, dz\\
& \qquad \qquad -\int\sb{\mathbb{R}^N}\frac{h|u_n|^{p^*(a,b)-2}u_nv}{|y|^{bp^*(a,b)}}\, dz
-\lambda \int\sb{\mathbb{R}^N}\frac{g|u_n|^{q-2}u_nv}{|y|^{cp^*(a,c)}}\, dz\Bigg).
\end{align*}
Consequently, 
 $\langle\varphi_\lambda'(u),v\rangle=0$ 
for every function $v \in \mathcal{D}\sb{a}\sp{1,p}(\mathbb{R}\sp{N} \backslash \{|y|=0\})$; hence, $\varphi_\lambda'(u)=0$. The last claim of the lemma follows directly from inequality~\eqref{bouchekif2.3B}. The lemma is proved. 
\end{proof}

The proof of the next result is partially inspired by
Ghergu and R{\u{a}}dulescu~\cite{MR2179582}. Before we state it, we define the number
\begin{align*}
d_\lambda^* \equiv \left(\frac{1}{p}-\frac{1}{p^*(a,b)}\right)h_0^{-\frac{p}{p^*(a,b)-p}}K(N,p,\mu,a,b)^{-\frac{p^*(a,b)}{p^*(a,b)-p}}-C_0\lambda^\frac{p}{p-q}.
\end{align*}
\begin{lema}\label{bouchekiflema4.2}
Suppose that the
hypotheses~\eqref{hipoteseg1},~\eqref{hipoteseg2},~\eqref{hipoteseh1}, 
and~\eqref{hipoteseh2} are valid. 
Let $(u_n)_{n\in\mathbb{N}} \subset
\mathcal{D}\sb{a}\sp{1,p}(\mathbb{R}\sp{N} \backslash \{|y|=0\})$ 
be a Palais-Smale sequence for the functional $\varphi_\lambda$ at the level $d < d_\lambda^*$. 
Then there exists a subsequence, still denoted in the same way, 
and there exists
$u \in \mathcal{D}\sb{a}\sp{1,p}(\mathbb{R}\sp{N} \backslash \{|y|=0\})$ 
such that
$u_n \to u$ strongly in 
$\mathcal{D}\sb{a}\sp{1,p}(\mathbb{R}\sp{N} \backslash \{|y|=0\})$ 
as $n \to \infty$.
\end{lema}

\begin{proof}
Using a standard argument we can prove the boundedness of the sequence 
$(u_n)_{n\in\mathbb{N}} \subset
\mathcal{D}\sb{a}\sp{1,p}(\mathbb{R}\sp{N} \backslash \{|y|=0\})$.
By the reflexivity of this Sobolev space, 
we can suppose, up to a subsequence
still denoted in the same way,
that there exists a function
$u\in \mathcal{D}\sb{a}\sp{1,p}(\mathbb{R}\sp{N} \backslash \{|y|=0\})$
such that, as \( n \to \infty \),
\begin{enumerate}[label={\upshape(\arabic*)}, align=left, widest=4, leftmargin=*]
\item
$u_n \rightharpoonup u$ weakly in $\mathcal{D}\sb{a}\sp{1,p}(\mathbb{R}\sp{N} \backslash \{|y|=0\})$;
\item
$u_n \rightharpoonup u$ weakly in $L_b^{p^*(a,b)}(\mathbb{R}^N\backslash\{|y|=0\})$;
\item
$u_n \rightharpoonup u$ weakly in $L_c^{q}(\mathbb{R}^N\backslash\{|y|=0\})$, 
for $ 1 < q < p $ and $ a \leqslant c < a + 1 $;
\item
$u_n \to u_1$ strongly in $L_s^{p^*(a,s)}(\mathbb{R}^N\backslash\{|y|=0\})$ for $a < s \leq a+1$;
\item
$u_n \to u$ a.\@~e.\@~in $\mathbb{R}^N$.
\end{enumerate}
Hence, the function $u$ is a weak solution to problem~\eqref{problemamultiplicidade} and the proof follows the same 
steps already made in the proof of 
Proposition~\ref{bouchekifprop3.2}.

Now we must show that the sequence 
$(u_n)_{n\in\mathbb{N}}$ converges strongly to the function $u$ in the Sobolev space $\mathcal{D}\sb{a}\sp{1,p}(\mathbb{R}\sp{N} \backslash \{|y|=0\})$.
To do this, let 
$(v_n)_{n\in\mathbb{N}} 
\subset \mathcal{D}\sb{a}\sp{1,p}(\mathbb{R}\sp{N} \backslash \{|y|=0\})$
be a sequence defined by $v_n=u_n - u$. 
By Br\'{e}zis-Lieb's lemma, we have 
\begin{align*}
\|v_n\|^p = \|u_n\|^p - \|u\|^p + o(1).
\end{align*}

\begin{afirmativa}\label{afirmativaparabouchekif4.2}
It is valid the relation
\begin{align*}
\int\sb{\mathbb{R}^N}\frac{h|v_n|^{p^*(a,b)}}{|y|^{bp^*(a,b)}}\, dz
= \int\sb{\mathbb{R}^N}\frac{h|u_n|^{p^*(a,b)}}{|y|^{bp^*(a,b)}}\, dz
-\int\sb{\mathbb{R}^N}\frac{h|u|^{p^*(a,b)}}{|y|^{bp^*(a,b)}}\, dz + o(1).
\end{align*}
\end{afirmativa}
\begin{proof}[Proof of the claim]
Let $\epsilon >0$ be given. 
By hypotheses~\eqref{hipoteseh1} and~\eqref{hipoteseh2} 
we can choose numbers $R_\epsilon > r_\epsilon > 0$ such that
\begin{align}\label{gherguradulescu11}
\int\sb{|y|<r_\epsilon}\frac{h|u|^{p^*(a,b)}}{|y|^{bp^*(a,b)}}\, dz < \epsilon
\quad \textrm{and} \quad
\int\sb{|y|>R_\epsilon}\frac{h|u|^{p^*(a,b)}}{|y|^{bp^*(a,b)}}\, dz < \epsilon.
\end{align}

Denoting 
$\Omega_\epsilon = 
\overline{B_{R_\epsilon}(0)}\backslash B_{r_\epsilon}(0)$,
it follows that
\begin{align*}
&\left|\int\sb{\mathbb{R}^N}\frac{h(|u_n|^{p^*(a,b)}-|u|^{p^*(a,b)}-|v_n|^{p^*(a,b)})}{|y|^{bp^*(a,b)}}\, dz \right|\\
& \qquad \leqslant \left|\int\sb{\Omega_\epsilon}\frac{h(|u_n|^{p^*(a,b)}-|u|^{p^*(a,b)})}{|y|^{bp^*(a,b)}}\, dz \right|
+\int\sb{\Omega_\epsilon}\frac{h|v_n|^{p^*(a,b)}}{|y|^{bp^*(a,b)}}\, dz\\
&\qquad \qquad + \int\sb{|y|<r_\epsilon}\frac{h|u|^{p^*(a,b)}}{|y|^{bp^*(a,b)}}\, dz 
+ \left|\int\sb{|y|<r_\epsilon}\frac{h(|u_n|^{p^*(a,b)}-|v_n|^{p^*(a,b)})}{|y|^{bp^*(a,b)}}\, dz\right|\\
&\qquad \qquad + \int\sb{|y|>R_\epsilon}\frac{h|u|^{p^*(a,b)}}{|y|^{bp^*(a,b)}}\, dz
+ \left|\int\sb{|y|>R_\epsilon}\frac{h(|u_n|^{p^*(a,b)}-|v_n|^{p^*(a,b)})}{|y|^{bp^*(a,b)}}\, dz\right|.
\end{align*}

Applying Lagrange's mean value theorem, 
there exists $\theta=\theta(z) \in [0,1]$ such that
\begin{align*}
\int\sb{|y|<r_\epsilon}\frac{h(|u_n|^{p^*(a,b)}-|v_n|^{p^*(a,b)})}{|y|^{bp^*(a,b)}}\, dz
= p^*(a,b)\int\sb{|y|<r_\epsilon}\frac{h|\theta u +v_n|^{p^*(a,b)-1}|u|}{|y|^{bp^*(a,b)}}\, dz.
\end{align*}
Using the first 
inequality in~\eqref{gherguradulescu11} together with inequality
$(x+y)^s \leqslant C(x^s + y^s)$ for any $x,y \in \mathbb{R}_*^+$, as well as
 H\"{o}lder's inequality, we deduce that 
\begin{align*}
&\int\sb{|y|<r_\epsilon}\frac{h|\theta u +v_n|^{p^*(a,b)-1}|u|}{|y|^{bp^*(a,b)}}\, dz\\
 &\qquad \leqslant C\int\sb{|y|<r_\epsilon}\frac{h(|u|^{p^*(a,b)} +|v_n|^{p^*(a,b)-1}|u|)}{|y|^{bp^*(a,b)}}\, dz\nonumber\\
& \qquad = C\int\sb{|y|<r_\epsilon}\frac{h|u|^{p^*(a,b)}}{|y|^{bp^*(a,b)}}\, dz
+ C\int\sb{|y|<r_\epsilon}\frac{h|v_n|^{p^*(a,b)-1}|u|}{|y|^{bp^*(a,b)}}\, dz\nonumber\\
& \qquad \leqslant C\epsilon 
+ C\left(\int\sb{|y|<r_\epsilon}\frac{h|v_n|^{p^*(a,b)}}{|y|^{bp^*(a,b)}}\, dz \right)^\frac{p^*(a,b)-1}{p^*(a,b)}\left(\int\sb{|y|<r_\epsilon}\frac{h|u|^{p^*(a,b)}}{|y|^{bp^*(a,b)}}\, dz \right)^\frac{1}{p^*(a,b)}\nonumber\\
& \qquad \leqslant C(\epsilon + \epsilon^\frac{1}{p^*(a,b)}).
\end{align*}
Hence, 
\begin{align*}
\int\sb{|y|<r_\epsilon}\frac{h|u|^{p^*(a,b)}}{|y|^{bp^*(a,b)}}\, dz
+ \left|\int\sb{|y|<r_\epsilon}\frac{h(|u_n|^{p^*(a,b)}-|v_n|^{p^*(a,b)})}{|y|^{bp^*(a,b)}}\, dz \right|\leqslant p^*(a,b)C(\epsilon + \epsilon^\frac{1}{p^*(a,b)})
\end{align*}
for some constant $C \in \mathbb{R}^+$ independent of
$n\in\mathbb{N}$ and 
of $\epsilon \in \mathbb{R}^+$.
In a similar way, we have
\begin{align*}
\int\sb{|y|>R_\epsilon}\frac{h|u|^{p^*(a,b)}}{|y|^{bp^*(a,b)}}\, dz
+ \left|\int\sb{|y|>R_\epsilon}\frac{h(|u_n|^{p^*(a,b)}-|v_n|^{p^*(a,b)})}{|y|^{bp^*(a,b)}}\, dz \right|
 \leqslant p^*(a,b)C(\epsilon + \epsilon^\frac{1}{p^*(a,b)}).
\end{align*}

Since $u_n \rightharpoonup u$ weakly in 
$\mathcal{D}\sb{a}\sp{1,p}(\mathbb{R}\sp{N} \backslash \{|y|=0\})$ 
as $n\to \infty$, 
the weak convergence in the Lebesgue space $L_b^{p^*(a,b)}(\mathbb{R}^N\backslash\{|y|=0\})$ imply 
\begin{align*}
\lim_{n\to\infty} \int\sb{\Omega_\epsilon}\frac{h(|u_n|^{p^*(a,b)}-|u|^{p^*(a,b)})}{|y|^{bp^*(a,b)}}\, dz=0
\quad \textrm{and} \quad
\lim_{n\to\infty} \int\sb{\Omega_\epsilon}\frac{h|v_n|^{p^*(a,b)}}{|y|^{bp^*(a,b)}}\, dz=0.
\end{align*}
Consequently, combining the two previous inequalities with the two previous limits, we obtain
\begin{align*}
\limsup_{n\to \infty}\left|\int\sb{\mathbb{R}^N}\frac{h(|u_n|^{p^*(a,b)}-|u|^{p^*(a,b)}-|v_n|^{p^*(a,b)})}{|y|^{bp^*(a,b)}}\, dz \right|
\leqslant (p^*(a,b)C+1)(\epsilon + \epsilon^\frac{1}{p^*(a,b)}),
\end{align*}
for some constant $C \in \mathbb{R}^+$ independent of
$n \in \mathbb{N}$ and of $\epsilon \in \mathbb{R}^+$. Finally, 
since $\epsilon > 0$ is arbitrary, it follows that
\begin{align*}
\lim_{n\to\infty} \int\sb{\mathbb{R}^N}\frac{h(|u_n|^{p^*(a,b)}-|v_n|^{p^*(a,b)})}{|y|^{bp^*(a,b)}}\, dz
=  \int\sb{\mathbb{R}^N}\frac{h|u|^{p^*(a,b)}}{|y|^{bp^*(a,b)}}\, dz.
\end{align*}
This concludes the proof of the claim.
\end{proof}

On the other hand, we also have the following result.

\begin{afirmativa}\label{afirmativaparabouchekif4.3}
It is valid the limit
\begin{align*}
\lim_{n\to\infty} \int\sb{\mathbb{R}^N}\frac{h(y)|v_n|^{p^*(a,b)}}{|y|^{bp^*(a,b)}}\, dz 
=\lim_{n\to\infty} h_0\int\sb{\mathbb{R}^N}\frac{|v_n|^{p^*(a,b)}}{|y|^{bp^*(a,b)}}\, dz.
\end{align*}
\end{afirmativa}
\begin{proof}[Proof of the claim]
Let $\epsilon>0$ be given.
By hypotheses~\eqref{hipoteseh1} and~\eqref{hipoteseh2} 
we can choose numbers $R_\epsilon > r_\epsilon > 0$ such that
\begin{align*}
|h(y)-h_0| = h(y) - h_0 < \epsilon \quad \textrm{a.\@~e.\@~in } 
\mathbb{R}^k\backslash 
\{\overline{B_{R_\epsilon}(0)}\backslash B_{r_\epsilon}(0)\}.
\end{align*}

Denoting $\Theta_\epsilon  = \mathbb{R}^{N-k}\times \{\overline{B_{R_\epsilon}(0)}\backslash B_{r_\epsilon}(0)\}$, we have
\begin{align*}
\int\sb{\mathbb{R}^N}\frac{(h(y)-h_0)|v_n|^{p^*(a,b)}}{|y|^{bp^*(a,b)}}\, dz
&\leqslant \epsilon\int\sb{\mathbb{R}^N\backslash\Theta_\epsilon}\frac{|v_n|^{p^*(a,b)}}{|y|^{bp^*(a,b)}}\, dz
+ (\|h\|_{L^\infty}-h_0)\int\sb{\Theta_\epsilon}\frac{|v_n|^{p^*(a,b)}}{|y|^{bp^*(a,b)}}\, dz\\
& \leqslant \epsilon\int\sb{\mathbb{R}^N}\frac{|v_n|^{p^*(a,b)}}{|y|^{bp^*(a,b)}}\, dz
+ (\|h\|_{L^\infty}-h_0)\int\sb{\Theta_\epsilon}\frac{|v_n|^{p^*(a,b)}}{|y|^{bp^*(a,b)}}\, dz.
\end{align*}

Since $v_n \rightharpoonup 0$ weakly in 
$\mathcal{D}\sb{a}\sp{1,p}(\mathbb{R}\sp{N} \backslash \{|y|=0\})$ 
as $n\to \infty$, Maz'ya's inequality implies that the sequence
$(v_n)_{n\in \mathbb{N}} 
\subset \mathcal{D}\sb{a}\sp{1,p}(\mathbb{R}\sp{N} \backslash \{|y|=0\})$
is bounded in $L_b^{p^*(a,b)}(\mathbb{R}^N\backslash\{|y|=0\})$. 
Moreover, by the weak convergence 
$u_n \rightharpoonup u$ in $L_b^{p^*(a,b)}(\mathbb{R}^N\backslash\{|y|=0\})$ 
already proved, it follows that
$v_n \to 0$ strongly in 
$L_{b,\textrm{loc}}^{p^*(a,b)}(\mathbb{R}^N\backslash\{|y|=0\})$ 
as $n\to \infty$. 
The previous relations imply that
\begin{align*}
\int\sb{\mathbb{R}^N}\frac{(h(y)-h_0)|v_n|^{p^*(a,b)}}{|y|^{bp^*(a,b)}}\, dz
\leqslant C\epsilon
\end{align*}
for some constant $C \in \mathbb{R}^+$ independent of $n \in \mathbb{N}$ 
and of $\epsilon \in \mathbb{R}^+$. 
Since $\epsilon >0$ is arbitrary, the claim follows.
\end{proof}

Following up, 
since $\varphi_\lambda(u_n) = d + o(1)$, 
by Br\'{e}zis-Lieb's lemma and by Claim~\ref{afirmativaparabouchekif4.2}, 
we get
\begin{align}\label{bouchekif4.4a}
&-\left(\frac{1}{q}-\frac{1}{p}\right)\|v_n\|^p
+ \left(\frac{1}{q}-\frac{1}{p^*(a,b)}\right)\int\sb{\mathbb{R}^N}\frac{h|v_n|^{p^*(a,b)}}{|y|^{bp^*(a,b)}}\, dz\nonumber\\
&\qquad \qquad = -\left(\frac{1}{q}-\frac{1}{p}\right)\|u_n\|^p
+ \left(\frac{1}{q}-\frac{1}{p^*(a,b)}\right)\int\sb{\mathbb{R}^N}\frac{h|u_n|^{p^*(a,b)}}{|y|^{bp^*(a,b)}}\, dz\nonumber\\
&\qquad \qquad \qquad +\left(\frac{1}{q}-\frac{1}{p}\right)\|u\|^p
- \left(\frac{1}{q}-\frac{1}{p^*(a,b)}\right)\int\sb{\mathbb{R}^N}\frac{h|u|^{p^*(a,b)}}{|y|^{bp^*(a,b)}}\, dz + o(1)\nonumber\\
&\qquad \qquad = \varphi_\lambda(u_n) - \varphi_\lambda(u) + o(1)\nonumber\\
&\qquad \qquad = d - \varphi_\lambda(u) + o(1).
\end{align}

Again by Br\'{e}zis-Lieb's lemma, by Claim~\ref{afirmativaparabouchekif4.2}, as well as the facts that 
$\varphi_\lambda'(u_n) = o(1)$ 
and $\displaystyle\lim_{n\to\infty}\int\sb{\mathbb{R}^N}\frac{g|u_n|^{q}}{|y|^{cp^*(a,c)}}\, dz=\displaystyle\int\sb{\mathbb{R}^N}\frac{g|u|^{q}}{|y|^{cp^*(a,c)}}\, dz$, we deduce that
\begin{align*}
\|v_n\|^p - \int\sb{\mathbb{R}^N}\frac{h|v_n|^{p^*(a,b)}}{|y|^{bp^*(a,b)}}\, dz
&= \langle \varphi_\lambda'(u_n),u_n\rangle - \langle \varphi_\lambda'(u),u\rangle + o(1) = o(1).
\end{align*}

By Claim~\ref{afirmativaparabouchekif4.2} and by the fact that $\displaystyle\int\sb{\mathbb{R}^N}
h |y|^{-bp^*(a,b)} |u_n|^{p^*(a,b)}\, dz$ 
and
$\displaystyle\int\sb{\mathbb{R}^N}
h |y|^{-bp^*(a,b)} |u|^{p^*(a,b)}\, dz$ 
are bounded, it follows that
$\displaystyle\int\sb{\mathbb{R}^N}
h |y|^{bp^*(a,b)} |v_n|^{p^*(a,b)} \, dz$ 
is also bounded. Hence, by the previous equation it follows that 
the sequence $(\|v_n\|)\sb{n \in \mathbb{N}} \subset \mathbb{R}$ is bounded. 
Applying Claim~\ref{afirmativaparabouchekif4.3} we can suppose that, as $n \to \infty$,
\begin{align*}
\|v_n\|^p \to l \quad \textrm{and} \quad \int\sb{\mathbb{R}^N}\frac{h|v_n|^{p^*(a,b)}}{|y|^{bp^*(a,b)}}\, dz
= h_0\int\sb{\mathbb{R}^N}\frac{|v_n|^{p^*(a,b)}}{|y|^{bp^*(a,b)}}\, dz \to l.
\end{align*}
If $l>0$, by Maz'ya's inequality it follows that
\begin{align*}
l\geqslant K(N,p,\mu,a,b)^{-\frac{p^*(a,b)}{p^*(a,b)-p}}h_0^{-\frac{p}{p^*(a,b)-p}};
\end{align*}
and by equation~\eqref{bouchekif4.4a}, by the definitions of $d_\lambda^*$ and $l$, 
by Lemma~\ref{bouchekiflema4.1}, and 
by the previous inequality, 
it follows that $d \geqslant d_\lambda^*$, 
which is a contradiction with the hypothesis $d<d_\lambda^*$.
Hence $l=0$, and this implies that $\|u_n -u\|^p \to 0$, that is, 
$u_n \to u$ in 
$\mathcal{D}\sb{a}\sp{1,p}(\mathbb{R}\sp{N} \backslash \{|y|=0\})$ 
as $n \to \infty$; besides, by equation~\eqref{bouchekif4.4a} we get
$d=\lim_{n\to\infty}\varphi_\lambda(u_n) = \varphi_\lambda(u)$.
\end{proof}

Contrary to the analysis made by 
Hsu~\cite{MR2559275} and also by
Hsu and Lin~\cite{MR2903809} for the case $p=2$, in the general situation 
$1 < p < N $ we do not have the explicit solutions for the optimal constant~\eqref{problemagazzinimusina}.
Therefore, we can not follow their steps to proceed the blow-up analysis.
In the next result, we follow closely the ideas by
Cao, Li and Zhou~\cite{MR1313196} 
and by Bouchekif and El Mokhtar~\cite{MR2801239}.

\begin{lema}\label{bouchekiflema4.3}
Suppose that the
hypotheses~\eqref{hipoteseg1},~\eqref{hipoteseg2},~\eqref{hipoteseh1}, 
and~\eqref{hipoteseh2} are valid. Then there exist $\Lambda_* > 0$
and $v \in \mathcal{D}\sb{a}\sp{1,p}(\mathbb{R}\sp{N} \backslash \{|y|=0\})$ 
such that for every $\lambda \in (0,\Lambda_*)$ it is valid the inequality
\begin{align*}
\sup_{t\geqslant 0} \varphi_\lambda(tv) < d_\lambda^*.
\end{align*}
In particular, $d^-<d_\lambda^*$ for every $\lambda\in(0,\Lambda_*)$.
\end{lema}

\begin{proof}
Using a result by Bhakta~\cite{MR2934676}, 
there exists a function 
$w \in \mathcal{D}\sb{a}\sp{1,p}(\mathbb{R}\sp{N} \backslash \{|y|=0\})$ 
that assumes the infimum~\eqref{problemagazzinimusina}. 
Let $ (x,y), (x_0,y_0) \in \mathbb{R}^{N-k}\times\mathbb{R}^k$; let the function
\( w_\epsilon \in \mathcal{D}\sb{a}\sp{1,p}(\mathbb{R}\sp{N} \backslash \{|y|=0\})\) 
be defined by
\begin{align*}
w_\epsilon (x,y) \equiv \epsilon\sp{-[N-p(a+1)]/p}\, 
u \left( x_0 + x/\epsilon , y/\epsilon \right).
\end{align*}
The three integrals involved in the definition of the infimum~\eqref{problemagazzinimusina}
are invariant by the action of this group of transformations; for that reason, the function $w_\epsilon$ also assumes the 
infimum~\eqref{problemagazzinimusina}. 
Moreover, we can suppose without loss of generality that
\begin{align}\label{relacoesnorma}
\|w_\epsilon\|^p = \|w_\epsilon\|_{L^{p^*(a,b)}_b(\mathbb{R}^N)}^{p^*(a,b)} =  K(N,p,\mu,a,b)^{-\frac{p^*(a,b)}{p^*(a,b)-p}}.
\end{align}

Using hypothesis~\eqref{hipoteseg2} we consider the function
$\psi_\epsilon \colon \mathbb{R}^N \to \mathbb{R}$ defined by
\begin{align*}
\psi_\epsilon(z) & \equiv
\begin{cases}
w_\epsilon(z),     
& \text{if $g(z)\geqslant 0$ for all $z \in \mathbb{R}^N$};\\
w_\epsilon(z-z_0), 
& \text{if there exist $z_0,\, z_1 \in \mathbb{R}^N$ such that $g(z_0) > 0$ 
and $g(z_1)<0$.}
\end{cases}
\end{align*}

\begin{afirmativa}\label{positividadedotermocomg}
There exists $\epsilon_0 > 0 $ such that
\begin{align}\label{bouchekif4.6}
\lambda\int\sb{\mathbb{R}^N}\frac{g|\psi_\epsilon|^{q}}{|y|^{cp^*(a,c)}}\, dz > 0
\quad \text{for every } \epsilon \in  (0,\epsilon_0).
\end{align}
\end{afirmativa}
\begin{proof}[Proof of the claim]
If $g(z)\geqslant 0$ for every $z \in \mathbb{R}^N$, 
then by the definiton of $\psi_\epsilon(z)$ 
the inequality~\eqref{bouchekif4.6} is valid 
independently of $\epsilon_0$. 
If there exist $z_0,\, z_1 \in \mathbb{R}^N$ such that 
$g(z_0) > 0$ and $g(z_1) < 0$, 
then by hypothesis~\eqref{hipoteseg2} 
there exist $\nu_0>0$ and $r_0 >0$ such that $g(z) >\nu_0$ 
for every $z \in B_{r_0}(0)$. 
By the above mentioned invariance properties of the integrals of $w_\epsilon$ involved in the definition of the infimum~\eqref{problemagazzinimusina},
this function concentrates in the ball 
$B_{r_\epsilon}(z_0)$ for $\epsilon_0=\epsilon_0(r_0)$ small enough. And this implies that 
\begin{align*}
\lambda\int\sb{\mathbb{R}^N}\frac{g|\omega_\epsilon(z-z_0)|^{q}}{|y|^{cp^*(a,c)}}\, dz 
= \lambda\int\sb{\mathbb{R}^N}\frac{g|\psi_\epsilon|^{q}}{|y|^{cp^*(a,c)}}\, dz> 0
\quad \textrm{for every } \epsilon \in  (0,\epsilon_0),
\end{align*}
which concludes the proof of the claim.
\end{proof}

Now we define the functions
$f_1, f_2 \colon \mathbb{R}_*^+ \to \mathbb{R}$,
respectively, by
\begin{align*}
f_1(t) \equiv \varphi_\lambda(t\psi_\epsilon)
\quad \textrm{and} \quad
f_2(t) \equiv \frac{t^p}{p}\|\psi_\epsilon\|^p
-\frac{t^{p^*(a,b)}}{p^*(a,b)}\int\sb{\mathbb{R}^N}\frac{h|\psi_\epsilon|^{p^*(a,b)}}{|y|^{bp^*(a,b)}}\, dz.
\end{align*}
Then, by hypothesis~\eqref{hipoteseh2} and since $\psi_\epsilon$ assumes
the infimum~\eqref{problemagazzinimusina}, it follows that
\begin{align}\label{bouchekif4.6b}
\max_{t\geqslant 0}f_2(t)
&=\left(\frac{1}{p}-\frac{1}{p^*(a,b)}\right)
\left(\|\psi_\epsilon\|^{p}\left(\int\sb{\mathbb{R}^N}\frac{h|\psi_\epsilon|^{p^*(a,b)}}{|y|^{bp^*(a,b)}}\, dz\right)^{-\frac{p}{p^*(a,b)}}\right)^\frac{p^*(a,b)}{p^*(a,b)-p}\nonumber\\
&\leqslant \left(\frac{1}{p}-\frac{1}{p^*(a,b)}\right)h_0^{-\frac{p}{p^*(a,b)-p}}K(N,p,\mu,a,b)^{-\frac{p^*(a,b}{p^*(a,b)-p}}.
\end{align}

\begin{afirmativa}\label{caoremark2.2}
There exists $\Lambda_2 > 0$ such that for every
$\lambda \in (0,\Lambda_2)$, it is valid the inequality
\begin{align*}
-\lambda\left(\frac{1}{q}-\frac{1}{p^*(a,b)}\right)\left(1-\frac{q}{p}\right)K(N,p,\mu,a,c)^\frac{q}{p}\|g\|_{L\sb{\alpha}\sp{r}(\mathbb{R}\sp{N})}^\frac{p}{p-q}
\leqslant d^+<0.
\end{align*}
\end{afirmativa}

\begin{proof}[Proof of the claim]
By Lemma~\ref{bouchekiflema2.6}~\ref{bouchekiflema2.6(i)}, 
we have $d^+<0$ for every $\lambda \in (0,\Lambda_1)$. 
By Proposition~\ref{bouchekifprop3.2} there exists a solution 
$u_1 \in \mathcal{D}\sb{a}\sp{1,p}(\mathbb{R}\sp{N} \backslash \{|y|=0\})$ 
to problem~\eqref{problemamultiplicidade} such that 
$\varphi_\lambda(u_1) = d^+$. 
Using equation~\eqref{phineharig} 
and also H\"{o}lder's, Maz'ya's and Young's inequalities, we deduce that
\begin{align*}
d^+ 
& =\left(\frac{1}{p}-\frac{1}{p^*(a,b)}\right)\|u_1\|^p -\lambda\left(\frac{1}{q}-\frac{1}{p^*(a,b)}\right)\int\sb{\mathbb{R}^N}\frac{g|u_1|^{q}}{|y|^{cp^*(a,c)}}\, dz\\
&\geqslant \left[\left(\frac{1}{p}-\frac{1}{p^*(a,b)}\right)- \lambda\left(\frac{1}{q}-\frac{1}{p^*(a,b)}\right)\frac{q}{p}K(N,p,\mu,a,c)^\frac{q}{p}\right]\|u_1\|^p\\
& \qquad - \lambda\left(\frac{1}{q}-\frac{1}{p^*(a,b)}\right)\left(1-\frac{q}{p} \right)K(N,p,\mu,a,c)^\frac{q}{p}\|g\|_{L\sb{\alpha}\sp{r}(\mathbb{R}\sp{N})}^\frac{p}{p-q}.
\end{align*}
Hence, there exists $\Lambda_2 > 0$ such that the claim holds.
\end{proof}

Applying Claim~\ref{caoremark2.2}, there exists $\Lambda_2 > 0$ such that
for every $\lambda \in (0,\Lambda_2)$ it is valid the inequality
\begin{align*}
&\left(\frac{1}{p}-\frac{1}{p^*(a,b)}\right)h_0^{-\frac{p}{p^*(a,b)-p}}K(N,p,\mu,a,b)^{-\frac{p^*(a,b)}{p^*(a,b)-p}} + d^+\\
&\qquad \geqslant \left(\frac{1}{p}-\frac{1}{p^*(a,b)}\right)h_0^{-\frac{p}{p^*(a,b)-p}}K(N,p,\mu,a,b)^{-\frac{p^*(a,b)}{p^*(a,b)-p}}\\
& \qquad \qquad - \lambda\left(\frac{1}{q}-\frac{1}{p^*(a,b)}\right)\left(1-\frac{q}{p} \right)K(N,p,\mu,a,c)^\frac{q}{p}\|g\|_{L\sb{\alpha}\sp{r}(\mathbb{R}\sp{N})}^\frac{p}{p-q}\\
& \qquad  >0
\end{align*}

By the continuity of the function $f_1$ and by 
inequalities~\eqref{bouchekif4.6} and~\eqref{bouchekif4.6b},
there exists a number $t_0 >0$ small enough and, 
what is crucial to our analysis,
independent of $\lambda$ and of $g$, 
such that
\begin{align}\label{bouchekif4.6c}
\sup_{0 \leqslant t \leqslant t_0} f_1(t) 
&\leqslant \sup_{t\geqslant 0} f_2(t) - \sup_{0 \leqslant t \leqslant t_0}\frac{\lambda}{q}t^q\int\sb{\mathbb{R}^N}\frac{g|\psi_\epsilon|^{q}}{|y|^{cp^*(a,c)}}\, dz\nonumber\\
&\leqslant \left(\frac{1}{p}-\frac{1}{p^*(a,b)}\right)h_0^{-\frac{p}{p^*(a,b)-p}}K(N,p,\mu,a,b)^{-\frac{p^*(a,b)}{p^*(a,b)-p}}
- \frac{\lambda}{q}t_0^q\int\sb{\mathbb{R}^N}\frac{g|\psi_\epsilon|^{q}}{|y|^{cp^*(a,c)}}\, dz.
\end{align}

Now we are going to estimate
$\sup_{t \geqslant t_0} f_1(t)$. 
By hypothesis~\eqref{hipoteseh2} and using equation~\eqref{relacoesnorma}, we deduce that
\begin{align*}
f_1(t) 
&\leqslant \frac{t^p}{p}\|\psi_\epsilon\|^p
-\frac{t^{p^*(a,b)}}{p^*(a,b)}h_0\int\sb{\mathbb{R}^N}\frac{|\psi_\epsilon|^{p^*(a,b)}}{|y|^{bp^*(a,b)}}\, dz
-\lambda\frac{t^{q}}{q}\int\sb{\mathbb{R}^N}\frac{g|\psi_\epsilon|^{q}}{|y|^{cp^*(a,c)}}\, dz\\
&= \frac{t^p}{p}K(N,p,\mu,a,b)^{-\frac{p^*(a,b)}{p^*(a,b)-p}}
-\frac{t^{p^*(a,b)}}{p^*(a,b)}h_0K(N,p,\mu,a,b)^{-\frac{p^*(a,b)}{p^*(a,b)-p}}-\lambda\frac{t_0^{q}}{q}\int\sb{\mathbb{R}^N}\frac{g|\psi_\epsilon|^{q}}{|y|^{cp^*(a,c)}}\, dz.
\end{align*}
Hence, for every $\epsilon \in (0,\epsilon_0)$, we have
\begin{align*}
\sup_{t \geqslant t_0} f_1(t) 
&\leqslant \sup_{t \geqslant t_0} 
\left( \frac{t^p}{p} 
-\frac{t^{p^*(a,b)}}{p^*(a,b)}
h_0 \right) K(N,p,\mu,a,b)^{-\frac{p^*(a,b)}{p^*(a,b)-p}}
-\lambda\frac{t_0^{q}}{q}\int\sb{\mathbb{R}^N}\frac{g|\psi_\epsilon|^{q}}{|y|^{cp^*(a,c)}}\, dz\\
&\leqslant 
\sup_{t \geqslant 0} 
\left(\frac{t^p}{p}-\frac{t^{p^*(a,b)}}{p^*(a,b)}h_0\right)K(N,p,\mu,a,b)^{-\frac{p^*(a,b)}{p^*(a,b)-p}}
-\lambda\frac{t_0^{q}}{q}\int\sb{\mathbb{R}^N}\frac{g|\psi_\epsilon|^{q}}{|y|^{cp^*(a,c)}}\, dz.
\end{align*}

To estimate this supremum, we define the function
$f_3 \colon \mathbb{R}_*^+ \to \mathbb{R}$ by
\begin{align*}
f_3(t) 
&\equiv 
\left(\frac{t^p}{p}-\frac{t^{p^*(a,b)}}{p^*(a,b)}h_0\right)K(N,p,\mu,a,b)^{-\frac{p^*(a,b)}{p^*(a,b)-p}}.
\end{align*}
Hence, we deduce that
\begin{align*}
\max_{t\geqslant 0} f_3(t) = \left(\frac{1}{p}-\frac{1}{p^*(a,b)}\right)h_0^{-\frac{p}{p^*(a,b)-p}}K(N,p,\mu,a,b)^{-\frac{p^*(a,b)}{p^*(a,b)-p}}.
\end{align*}
Therefore,
\begin{align*}
\sup_{t \geqslant t_0} f_1(t) 
&\leqslant \left(\frac{1}{p}-\frac{1}{p^*(a,b)}\right)h_0^{-\frac{p}{p^*(a,b)-p}}K(N,p,\mu,a,b)^{-\frac{p^*(a,b)}{p^*(a,b)-p}}
- \frac{\lambda}{q}t_0^q\int\sb{\mathbb{R}^N}\frac{g|\psi_\epsilon|^{q}}{|y|^{cp^*(a,c)}}\, dz.
\end{align*}

Combining the previous inequality with inequality~\eqref{bouchekif4.6c}, 
we determine
\begin{align}\label{bouchekif4.6cfinal}
\sup_{t \geqslant 0} f_1(t) 
&\leqslant \left(\frac{1}{p}-\frac{1}{p^*(a,b)}\right)h_0^{-\frac{p}{p^*(a,b)-p}}K(N,p,\mu,a,b)^{-\frac{p^*(a,b)}{p^*(a,b)-p}}
- \frac{\lambda}{q}t_0^q\int\sb{\mathbb{R}^N}\frac{g|\psi_\epsilon|^{q}}{|y|^{cp^*(a,c)}}\, dz.
\end{align}

Following up, let $\lambda \in \mathbb{R}$ be chosen so that
\begin{align*}
0 < \lambda < \left(\frac{t_0^q}{qC_0}\sup_{0<\epsilon<\epsilon_0}\int\sb{\mathbb{R}^N}\frac{g|\psi_\epsilon|^{q}}{|y|^{cp^*(a,c)}}\, dz\right)^\frac{p-q}{q} \equiv \Lambda_3.
\end{align*}
As a consequence of this choice,
 for every $\lambda \in (0,\Lambda_3)$ we have
\begin{align}\label{bouchekif4.6d}
- \frac{\lambda}{q}t_0^q\sup_{0<\epsilon<\epsilon_0}\int\sb{\mathbb{R}^N}\frac{g|\psi_\epsilon|^{q}}{|y|^{cp^*(a,c)}}\, dz
< - C_0 \lambda^\frac{p}{p-q}.
\end{align}

Finally, we define
\begin{align*}
\Lambda_* \equiv \min\left\{\Lambda_1, \Lambda_2, \Lambda_3 \right\}.
\end{align*}
Therefore, by inequalities~\eqref{bouchekif4.6cfinal} and~\eqref{bouchekif4.6d}, 
and by the definition of $d_\lambda^*$, 
for every $\lambda\in(0,\Lambda_*)$ we obtain
\begin{align*}
\sup_{t \geqslant 0} \varphi_\lambda(t\psi_\epsilon)
&\leqslant \left(\frac{1}{p}-\frac{1}{p^*(a,b)}\right)h_0^{-\frac{p}{p^*(a,b)-p}}K(N,p,\mu,a,b)^{-\frac{p^*(a,b)}{p^*(a,b)-p}}
- \frac{\lambda}{q}t_0^q\int\sb{\mathbb{R}^N}\frac{g|\psi_\epsilon|^{q}}{|y|^{cp^*(a,c)}}\, dz\nonumber\\
&\leqslant \left(\frac{1}{p}-\frac{1}{p^*(a,b)}\right)h_0^{-\frac{p}{p^*(a,b)-p}}K(N,p,\mu,a,b)^{-\frac{p^*(a,b)}{p^*(a,b)-p}}
- C_0\lambda^\frac{p}{p-q}\nonumber\\
&= d_\lambda^*.
\end{align*}

It remains to prove that
$d^- < d_\lambda^*$ 
for every $\lambda\in (0,\Lambda_*)$.
By Claim~\ref{positividadedotermocomg}, by Lemma~\ref{bouchekiflema2.7}~\ref{bouchekiflema2.7(ii)},
by the definition of $d^-$, and using the previous inequality regarding $d_\lambda^*$, we obtain the existence of a sequence
$(t_{\epsilon_n})_{n\in \mathbb{N}} \subset \mathbb{R}^+$ such that $t_{\epsilon_n}\omega_{\epsilon_n}\in \mathcal{N}_\lambda^-$ and
\begin{align*}
d^- \leqslant \varphi_\lambda(t_{\epsilon_n}\omega_{\epsilon_n}) \leqslant \sup_{t\geqslant 0} \varphi_\lambda(t_{\epsilon_n}\omega_{\epsilon_n}) < d_\lambda^*,
\end{align*}
for every $\lambda \in (0,\Lambda_*)$. This concludes the proof of the lemma.
\end{proof}

Now we establish the existence of a minimum for the functional 
$\varphi_\lambda$ in $\mathcal{N}_\lambda^-$.

\begin{prop}\label{bouchekifprop4.4}
Let $\Lambda_0 \equiv \min\{(q/p)\Lambda_1,\Lambda_*\}$. 
If $\lambda \in  (0,\Lambda_0)$, then
the functional $\varphi_\lambda$ 
has a positive minimizer $u_2 \in \mathcal{N}_\lambda^-$ 
such that 
$\varphi_\lambda(u_2) = d^-$
and the function $u_2$ is a solution to 
problem~\eqref{problemamultiplicidade} in $\mathcal{D}\sb{a}\sp{1,p}(\mathbb{R}\sp{N} \backslash \{|y|=0\})$.
\end{prop}

\begin{proof}
By Proposition~\ref{hsuprop3.3}~\ref{hsuprop3.3(ii)},
if $\lambda\in(0,(q/p)\Lambda_1)$, then there exists a 
Palais-Smale sequence $(u_n)_{n \in \mathbb{N}}\subset \mathcal{N}_\lambda^-$ 
at the level $d^-$ for the functional $\varphi_\lambda$. 
From Lemmas~\ref{bouchekiflema4.2},~\ref{bouchekiflema4.3} 
and~\ref{bouchekiflema2.6}~\ref{bouchekiflema2.6(ii)},
for every $\lambda\in(0,\Lambda_*)$
the functional $\varphi_\lambda$ 
verifies the Palais-Smale condition at the level $d^- > 0$. 
Hence, the sequence $(u_n)_{n \in \mathbb{N}} $ 
is bounded in 
$\mathcal{D}\sb{a}\sp{1,p}(\mathbb{R}\sp{N} \backslash \{|y|=0\})$,
and there exists a subsequence, still denoted in the same way,
and there exists a function $u_2 \in \mathcal{N}_\lambda^-$
such that $u_n \to u_2$ strongly in 
$\mathcal{D}\sb{a}\sp{1,p}(\mathbb{R}\sp{N} \backslash \{|y|=0\})$
as $n\to \infty$ 
with $\varphi_\lambda(u_2) = d^-$. 
Finally, using the same arguments from the
proof of Proposition~\ref{bouchekifprop3.2} 
and the facts that 
$|u_2| \in \mathcal{N}_\lambda^-$ 
and $\varphi_\lambda(u_2) = \varphi_\lambda(|u_2|)$, 
we deduce that for every $\lambda \in (0,\Lambda_1)$
the function $u_2$ is a positive solution to
problem~\eqref{problemamultiplicidade}.
\end{proof}

\begin{proof}[Proof of Theorem~\ref{teoremamultiplicidade2}]
By Propositions~\ref{bouchekifprop3.2} and~\ref{bouchekifprop4.4}, 
we obtain the positive solutions $u_1$ and $u_2$ 
to problem~\eqref{problemamultiplicidade} such that
$u_1 \in \mathcal{N}_\lambda^+$ and $u_2 \in \mathcal{N}_\lambda^-$.
And since $\mathcal{N}_\lambda^+ \cap \mathcal{N}_\lambda^- = \emptyset$,
we conclude that $u_1$ and $u_2$ are distinct solutions.
\end{proof}

\bibliographystyle{abbrv}

\bibliography{bibtese}

\begin{thebibliography}{10}

\bibitem{MR2233146}
B.~Abdellaoui, V.~Felli, and I.~Peral.
\newblock Existence and nonexistence results for quasilinear elliptic equations
  involving the {$p$}-{L}aplacian.
\newblock {\em Boll. Unione Mat. Ital. Sez. B Artic. Ric. Mat. (8)},
  9(2):445--484, 2006.

\bibitem{MR1884469}
B.~Abdellaoui and I.~Peral.
\newblock Some results for semilinear elliptic equations with critical
  potential.
\newblock {\em Proc. Roy. Soc. Edinburgh Sect. A}, 132(1):1--24, 2002.

\bibitem{MR1276168}
A.~Ambrosetti, H.~Br{\'{e}}zis, and G.~Cerami.
\newblock Combined effects of concave and convex nonlinearities in some
  elliptic problems.
\newblock {\em J. Funct. Anal.}, 122(2):519--543, 1994.

\bibitem{ams:exist}
R.~B. Assun{\c{c}}{\~a}o, O.~H. Miyagaki, and W.~W. Santos.
\newblock Existence, regularity and non-existence results for a class of
  quasilinear elliptic problems with cylindrical singularities and multiple
  critical nonlinearities.
\newblock Preprint.

\bibitem{MR1918928}
M.~Badiale and G.~Tarantello.
\newblock A {S}obolev-{H}ardy inequality with applications to a nonlinear
  elliptic equation arising in astrophysics.
\newblock {\em Arch. Ration. Mech. Anal.}, 163(4):259--293, 2002.

\bibitem{MR2934676}
M.~Bhakta.
\newblock On the existence and breaking symmetry of the ground state solution
  of {H}ardy {S}obolev type equations with weighted {$p$}-{L}aplacian.
\newblock {\em Adv. Nonlinear Stud.}, 12(3):555--568, 2012.

\bibitem{MR2801239}
M.~Bouchekif and M.~E. M.~O. El~Mokhtar.
\newblock Nonhomogeneous elliptic equations with decaying cylindrical potential
  and critical exponent.
\newblock {\em Electron. J. Differential Equations}, 2011(54):1--10, 2011.

\bibitem{MR2480011}
M.~Bouchekif and A.~Matallah.
\newblock Multiple positive solutions for elliptic equations involving a
  concave term and critical {S}obolev-{H}ardy exponent.
\newblock {\em Appl. Math. Lett.}, 22(2):268--275, 2009.

\bibitem{MR768824}
L.~Caffarelli, R.~Kohn, and L.~Nirenberg.
\newblock First order interpolation inequalities with weights.
\newblock {\em Compositio Math.}, 53(3):259--275, 1984.

\bibitem{MR2092869}
D.~Cao and P.~Han.
\newblock Solutions for semilinear elliptic equations with critical exponents
  and {H}ardy potential.
\newblock {\em J. Differential Equations}, 205(2):521--537, 2004.

\bibitem{MR1313196}
D.~M. Cao, G.~B. Li, and H.~S. Zhou.
\newblock Multiple solutions for nonhomogeneous elliptic equations involving
  critical {S}obolev exponent.
\newblock {\em Proc. Roy. Soc. Edinburgh Sect. A}, 124(6):1177--1191, 1994.

\bibitem{MR2072017}
J.~Chen.
\newblock Multiple positive solutions for a class of nonlinear elliptic
  equations.
\newblock {\em J. Math. Anal. Appl.}, 295(2):341--354, 2004.

\bibitem{MR1036731}
R.~Dautray and J.-L. Lions.
\newblock {\em Mathematical analysis and numerical methods for science and
  technology. {V}ol. 1}.
\newblock Springer-Verlag, Berlin, 1990.

\bibitem{MR1876652}
A.~Ferrero and F.~Gazzola.
\newblock Existence of solutions for singular critical growth semilinear
  elliptic equations.
\newblock {\em J. Differential Equations}, 177(2):494--522, 2001.

\bibitem{MR2498753}
R.~Filippucci, P.~Pucci, and F.~Robert.
\newblock On a {$p$}-{L}aplace equation with multiple critical nonlinearities.
\newblock {\em J. Math. Pures Appl. (9)}, 91(2):156--177, 2009.

\bibitem{MR2179582}
M.~Ghergu and V.~R{\u{a}}dulescu.
\newblock Singular elliptic problems with lack of compactness.
\newblock {\em Ann. Mat. Pura Appl. (4)}, 185(1):63--79, 2006.

\bibitem{MR2865669}
M.~Ghergu and V.~D. R{\u{a}}dulescu.
\newblock {\em Nonlinear {PDE}s, Mathematical models in biology, chemistry and
  population genetics}.
\newblock Springer Monographs in Mathematics. Springer, Heidelberg, 2012.

\bibitem{MR2210661}
N.~Ghoussoub and F.~Robert.
\newblock Concentration estimates for {E}mden-{F}owler equations with boundary
  singularities and critical growth.
\newblock {\em IMRP Int. Math. Res. Pap.}, 2006:1--85, 2006.

\bibitem{MR1695021}
N.~Ghoussoub and C.~Yuan.
\newblock Multiple solutions for quasi-linear {PDE}s involving the critical
  {S}obolev and {H}ardy exponents.
\newblock {\em Trans. Amer. Math. Soc.}, 352(12):5703--5743, 2000.

\bibitem{MR2559275}
T.-S. Hsu.
\newblock Multiplicity results for {$p$}-{L}aplacian with critical nonlinearity
  of concave-convex type and sign-changing weight functions.
\newblock {\em Abstr. Appl. Anal.}, 2009(Art. ID 652109):1--24, 2009.

\bibitem{MR2802979}
T.-S. Hsu.
\newblock Multiple positive solutions for a quasilinear elliptic problem
  involving critical {S}obolev-{H}ardy exponents and concave-convex
  nonlinearities.
\newblock {\em Nonlinear Anal.}, 74(12):3934--3944, 2011.

\bibitem{MR2525571}
T.-S. Hsu and H.-L. Lin.
\newblock Multiple positive solutions for singular elliptic equations with
  concave-convex nonlinearities and sign-changing weights.
\newblock {\em Bound. Value Probl.}, 2009(Art. ID 584203):1--17, 2009.

\bibitem{MR2651376}
T.-S. Hsu and H.-L. Lin.
\newblock Multiple positive solutions for singular elliptic equations with
  weighted {H}ardy terms and critical {S}obolev-{H}ardy exponents.
\newblock {\em Proc. Roy. Soc. Edinburgh Sect. A}, 140(3):617--633, 2010.

\bibitem{MR2903809}
T.-S. Hsu and H.-L. Lin.
\newblock Multiplicity of positive solutions for weighted quasilinear elliptic
  equations involving critical {H}ardy-{S}obolev exponents and concave-convex
  nonlinearities.
\newblock {\em Abstr. Appl. Anal.}, 2012(Art. ID 579481):1--19, 2012.

\bibitem{MR2161873}
D.~Kang and S.~Peng.
\newblock Solutions for semilinear elliptic problems with critical
  {S}obolev-{H}ardy exponents and {H}ardy potential.
\newblock {\em Appl. Math. Lett.}, 18(10):1094--1100, 2005.

\bibitem{MR864416}
C.~S. Lin.
\newblock Interpolation inequalities with weights.
\newblock {\em Comm. Partial Differential Equations}, 11(14):1515--1538, 1986.

\bibitem{maz'ya}
V.~G. Maz'ya.
\newblock {\em Sobolev Spaces}.
\newblock Springer-Verlag, Berlin, 1980.

\bibitem{MR1990020}
S.~Secchi, D.~Smets, and M.~Willem.
\newblock Remarks on a {H}ardy-{S}obolev inequality.
\newblock {\em C. R. Math. Acad. Sci. Paris}, 336(10):811--815, 2003.

\bibitem{MR1168304}
G.~Tarantello.
\newblock On nonhomogeneous elliptic equations involving critical {S}obolev
  exponent.
\newblock {\em Ann. Inst. H. Poincar\'e Anal. Non Lin\'eaire}, 9(3):281--304,
  1992.

\bibitem{MR2733236}
L.~Wang, Q.~Wei, and D.~Kang.
\newblock Multiple positive solutions for {$p$}-{L}aplace elliptic equations
  involving concave-convex nonlinearities and a {H}ardy-type term.
\newblock {\em Nonlinear Anal.}, 74(2):626--638, 2011.

\bibitem{MR2606810}
B.~Xuan and J.~Wang.
\newblock Existence of a nontrivial weak solution to quasilinear elliptic
  equations with singular weights and multiple critical exponents.
\newblock {\em Nonlinear Anal.}, 72(9-10):3649--3658, 2010.

\end{thebibliography}

\end{document}